\documentclass[10pt]{article}
\usepackage[british]{babel}
\usepackage{geometry,amsmath,amssymb,enumerate,bbm,xcolor,latexsym,theorem}
\catcode`\<=\active \def<{
\fontencoding{T1}\selectfont\symbol{60}\fontencoding{\encodingdefault}}
\newcommand{\assign}{:=}
\newcommand{\backassign}{=:}
\newcommand{\mathD}{\mathrm{D}}
\newcommand{\mathd}{\mathrm{d}}
\newcommand{\nobracket}{}
\newcommand{\nosymbol}{}
\newcommand{\of}{:}
\newcommand{\rightarrowlim}{\mathop{\rightarrow}\limits}
\newcommand{\tmaffiliation}[1]{\\ #1}
\newcommand{\tmcolor}[2]{{\color{#1}{#2}}}
\newcommand{\tmem}[1]{{\em #1\/}}
\newcommand{\tmemail}[1]{\\ \textit{Email:} \texttt{#1}}
\newcommand{\tmop}[1]{\ensuremath{\operatorname{#1}}}
\newcommand{\tmtextbf}[1]{{\bfseries{#1}}}
\newcommand{\tmtextit}[1]{{\itshape{#1}}}
\newenvironment{enumerateroman}{\begin{enumerate}[i.] }{\end{enumerate}}
\newenvironment{proof}{\noindent\textbf{Proof\ }}{\hspace*{\fill}$\Box$\medskip}
\newenvironment{proof*}[1]{\noindent\textbf{#1\ }}{\hspace*{\fill}$\Box$\medskip}
\newtheorem{corollary}{Corollary}
\newtheorem{definition}{Definition}
\newtheorem{lemma}{Lemma}
\newtheorem{proposition}{Proposition}
{\theorembodyfont{\rmfamily}\newtheorem{remark}{Remark}}
\newtheorem{theorem}{Theorem}

\usepackage[T3,T1]{fontenc} \DeclareSymbolFont{tipa}{T3}{cmr}{m}{n}
\DeclareMathAccent{\invbreve}{\mathalpha}{tipa}{16}

\begin{document}

\title{Hyperviscous stochastic Navier--Stokes equations with white noise
invariant measure}

\author{
  M.~Gubinelli
  \tmaffiliation{Institute for Applied Mathematics \& Hausdorff Center for
  Mathematics\\
  University of Bonn, Germany}
  \tmemail{gubinelli@iam.uni-bonn.de}
  \and
  M.~Turra
  \tmaffiliation{Institute for Applied Mathematics \& Hausdorff Center for
  Mathematics\\
  University of Bonn, Germany}
  \tmemail{mattia.turra@iam.uni-bonn.de}
}

\maketitle

\begin{abstract}
  We prove existence and uniqueness of martingale solutions to a (slightly)
  hyper-viscous stochastic Navier--Stokes equation in 2d with initial
  conditions absolutely continuous with respect to the Gibbs measure
  associated to the energy, getting the results both in the torus and in the
  whole space setting.
\end{abstract}

\section{Introduction}

Consider the following stochastic hyper-viscous Navier--Stokes equation on
$\mathbb{R}_+ \times \mathbb{T}^2$
\begin{equation}
  \begin{array}{lll}
    \partial_t u & = & - A^{\theta} u - u \nosymbol \cdot \nabla u - \nabla p
    - \sqrt{2} \nabla^{\perp} A^{(\theta - 1) / 2} \xi,\\
    \tmop{div} u & = & 0,
  \end{array} \label{eq:navierstokes}
\end{equation}
where $\mathbb{T}^2$ is the two dimensional torus, $A = - \Delta$ on
$\mathbb{T}^2$, $\nabla^{\perp} \assign (\partial_2, - \partial_1)$, $\theta >
1$, and $\xi$ denotes a space-time white noise. The initial condition for $u$
will be taken distributed according to the white noise on $\mathbb{T}^2$ or an
absolute continuous perturbation thereof with density in $L^2$. The white
noise on $\mathbb{T}^2$ is formally invariant for the dynamics described
by~{\eqref{eq:navierstokes}} and the existence theory for the corresponding
stationary process has been addressed by Gubinelli and Jara
in~{\cite{gubinelli2013regularization}} using the concept of \tmtextit{energy
solutions} for any $\theta > 1$. Uniqueness was left open in the
aforementioned paper, and the main aim of the present work, which can be
thought of as a continuation of~{\cite{gubinelli2018infinitesimal}}, is to
introduce a martingale problem formulation~{\eqref{eq:navierstokes}} for which
we can prove uniqueness.

\

In order to properly formulate the martingale problem, we need to investigate
the infinitesimal generator for eq.~{\eqref{eq:navierstokes}} and uniqueness
will result from suitable solutions of the associated Kolmogorov backward
equation.

\

The variable $u$ appearing in eq.~{\eqref{eq:navierstokes}} represents
physically the {\tmem{velocity}} of a fluid. Rewriting the equation for the
\tmtextit{vorticity} $\omega \assign \nabla^{\perp} \cdot u$ yields
\begin{equation}
  \partial_t \omega = - A^{\theta} \omega - u \cdot \nabla \omega + \sqrt{2}
  A^{(\theta + 1) / 2} \xi . \label{eq:ns-vort}
\end{equation}
We also have the relation $u = K \ast \omega$, for the {\tmem{Biot-Savart
kernel}} $K$ on $\mathbb{T}^2$ given by
\[ K (x) = \frac{1}{2 \pi \iota}  \sum_{k \in \mathbb{Z}_0^2}
   \frac{k^{\perp}}{| k |^2 \nobracket \nobracket} e^{2 \pi \iota k \cdot x} =
   - \sum_{k \in \mathbb{Z}_0^2} \frac{2 \pi \iota k^{\perp}}{| 2 \pi k |^2
   \nobracket \nobracket} e^{2 \pi \iota k \cdot x}, \]
where $k^{\perp} = (k_2, - k_1)$ and $\mathbb{Z}_0^2 =\mathbb{Z}^2 \backslash
\{ 0 \}$. It is more convenient to work with the scalar quantity $\omega$ and
with eq.~{\eqref{eq:ns-vort}}.

\

The standard stochastic Navier-Stokes equation corresponds to the case
$\theta = 1$. However, this regime is quite singular for the white noise
initial condition and no results are known, not even existence of a stationary
solution, e.g.~from limit of Galerkin approximations. While a bit unphysical,
we will stick here to the \tmtextit{hyper-viscous regime}, namely $\theta >
1$. Note that the noise has to be coloured accordingly in order to preserve
the white noise as invariant measure. Moreover, we call \tmtextit{energy
measure} the law under which the velocity field is a (vector-valued,
incompressible) white noise. In terms of vorticity $\omega$, the kinetic
energy of the fluid configuration $u$ is
\[ \| u \|_{L^2}^2 = \int_{\mathbb{T}^2} | K \ast \omega |^2 (x) \mathd x =
   \sum_{k \in \mathbb{Z}_0^2} | \hat{K} (k) |^2 | \hat{\omega} (k) |^2 =
   \sum_{k \in \mathbb{Z}_0^2} \left| \frac{2 \pi \iota k^{\perp}}{| 2 \pi k
   |^2} \right|^2 | \hat{\omega} (k) |^2 = \| (- \Delta)^{- 1 / 2} \omega
   \|_{L^2}^2, \]
where $\hat{f} : \mathbb{Z}^2 \rightarrow \mathbb{C}$ denotes the Fourier
transform of $f : \mathbb{T}^2 \rightarrow \mathbb{R}$ defined as to have $f
(x) = \sum_{k \in \mathbb{Z}^2} e^{2 \pi \iota k \cdot x} \hat{f} (k)$. The
energy measure is thus formally given~by
\begin{equation}
  \mu (\mathd \omega) = \frac{1}{C} e^{- \frac{1}{2}  \| A^{- 1 / 2} \omega
  \|^2_{L^2}} \mathd \omega, \label{eq:energy-measure}
\end{equation}
where $\mathd \omega$ denotes the ``Lebesgue measure'' on functions
on~$\mathbb{T}^2$. Rigorously, this of course means the product Gaussian
measure
\[ \mu (\mathd \omega) = \prod_{k \in \mathbb{Z}_0^2} \frac{1}{C_k} \exp
   \left( - \frac{| \hat{\omega} (k) |^2}{2 | 2 \pi k |^2} \right) \mathd
   \hat{\omega} (k), \]
with the restriction that $\hat{\omega} (- k) = \overline{\hat{\omega} (k)}$.
For $f, g \in C^{\infty} (\mathbb{T}^2)$, we have
\[ \int \omega (f) \omega (g) \mu (\mathd \omega) = \sum_{k \in
   \mathbb{Z}_0^2} | 2 \pi k |^2  \overline{\hat{f} (k)}  \hat{g} (k) =
   \langle A^{1 / 2} f, A^{1 / 2} g \rangle_{L^2 (\mathbb{T}^2)} = \langle f,
   g \rangle_{H^1 (\mathbb{T}^2)} . \]
We can use the right-hand side as the definition of the covariance of $(\omega
(f))_{f \in C^{\infty} (\mathbb{T}^2)}$, which determines the law of $\omega$
as a centred Gaussian process indexed by $H^1 (\mathbb{T}^2)$. If $\eta$ is a
white noise on $L^2 (\mathbb{T}^2)$, then $\mu$ has the same distribution as
$A^{1 / 2} \eta$ and it is only supported on $H^{- 2 -} (\mathbb{T}^2)$.

\

A different situation occurs if we consider initial conditions distributed
according to the \tmtextit{enstrophy measure}, namely the Gaussian measure for
which the initial vorticity is a white noise. This measure is more regular
than the energy measure and more results are known, both for the Euler
dynamics (i.e.,~without dissipation and noise) and for the stochastic
Navier-Stokes dynamics, see
e.g.~{\cite{albeverio1990global,albeverio2004uniqueness,albeverio2008some}}.

\

As we already remarked, we use here the technique introduced
in~{\cite{gubinelli2018infinitesimal}} and strongly rooted in the notion of
energy solution of Gon{\c c}alves and Jara~{\cite{Goncalves2014}}, extended
in~{\cite{gubinelli2013regularization}}. With respect
to~{\cite{gubinelli2018infinitesimal}} we give a slightly different
formulation which simplifies certain technical estimates. The core of the
argument however remains the same. The main point is to consider the
well-posedness problem for~{\eqref{eq:navierstokes}} as a problem of
{\tmem{singular diffusion}}, i.e. diffusions with distributional drift. The
papers~{\cite{Flandoli2003,Flandoli2004,Delarue2016,Cannizzaro2018}} all
follow a similar strategy in order to identify a domain for the formal
infinitesimal generator $\mathcal{L}= \tfrac{1}{2} \Delta + b \cdot \nabla$ of
a finite dimensional diffusion. Then they show existence and uniqueness of
solutions for the corresponding martingale problem. The key difficulty is that
for distributional $b$ the domain does not contain any smooth functions and
instead one has to identify a class of non-smooth test functions with a
special structure, adapted to $b$. Roughly speaking they must be local
perturbations of a linear functional constructed from $b$. Recently other
results of regularisation by noise for SPDEs~{\cite{DaPrato2013,DaPrato2016}}
have been obtained. An important difference is that our drift is unbounded and
not even a function. The connection between energy solutions and
regularisation by noise was first observed
in~{\cite{gubinelli2013regularization}}.

\paragraph*{Plan of the paper}In Section~\ref{s:galerkin} we introduce a
Galerkin approximation for the nonlinearity $u \cdot \nabla \omega$ and study
the infinitesimal generator of the approximating equation. The martingale
problem for cylinder function related to eq.~{\eqref{eq:ns-vort}} is
introduced in Section~\ref{sec:cyl-mart}. In Section~\ref{sec:uniq} we prove
uniqueness for the martingale problem via existence of classical solutions to
the backward Kolmogorov equation for the operator $\mathcal{L}$ involved in
the martingale problem. The construction of a domain to such an operator is
the core of the work and it will be tackled in Section~\ref{s:kolmogorov},
where we provide also existence and uniqueness for the associated Kolmogorov
equation. In Section~\ref{sec:bounds} we show some crucial bounds on the
drift. Finally, Section~\ref{s:whole-space} extends the results obtained in
the previous part of the paper to the whole space case, that is, to the
hyper-viscous stochastic Navier--Stokes equation on~$\mathbb{R}^2$.
Appendix~\ref{s:app-A} contains some auxiliary results.

\paragraph{Notation}Let us fix here some notation that will be adopted
throughout the paper. The Schwartz space on $\mathbb{T}^2$ is denoted by
$\mathcal{S} (\mathbb{T}^2)$ and its dual $\mathcal{S}' (\mathbb{T}^2)$ is the
space of tempered distributions. We denote by $H^s (\mathbb{T}^2)$, $s \in
\mathbb{R}$, the completion of the space of functions $f \in \mathcal{S}
(\mathbb{T}^2)$ such that
\[ \| f \|_{H^s (\mathbb{T}^2)}^2 \assign \int_{\mathbb{T}^2} | z |^{2 s} |
   \hat{f} (z) |^s \mathd z < + \infty, \]
identifying $f$ and $g$ whenever $\| f - g \|_{H^s (\mathbb{T}^2)} = 0$. From
now on, we write $C (X, Y)$ to indicate the space of continuous functions from
$X$ to~$Y$. We also write $a \lesssim b$ or $b \gtrsim a$ if there exists a
constant $C > 0$, independent of the variables under consideration, such that
$a \leqslant C \cdot b$, and $a \simeq b$ if both $a \lesssim b$ and $b
\lesssim a$. If the aforementioned constant $C$ depends on a variable, say $C
= C (x)$, then we use the notation $a \lesssim_x b$, and similarly
for~$\gtrsim$. For the sake of brevity, we will also use the notation $k_{1 :
n} = (k_1, \ldots, k_n)$.

\section{Galerkin approximations}\label{s:galerkin}

In order to rigorously study the eq.~{\eqref{eq:ns-vort}}, consider the
solution $(\omega^m_t)_{t \geqslant 0}$ to its Galerkin approximation:
\begin{equation}
  \partial_t \omega^m = - A^{\theta} \omega^m - B_m (\omega^m) + \sqrt{2}
  A^{(1 + \theta) / 2} \xi, \label{eq:ns-vort-appr}
\end{equation}
where
\[ B_m (\omega) \assign \tmop{div} \Pi_m ((K \ast \Pi_m \omega) \Pi_m \omega),
\]
and $\Pi_m$ denotes the projection onto Fourier modes of size less than $m$,
namely $\Pi_m f (x) = \sum_{| k | \leqslant m} e^{2 \pi \iota k \cdot x}
\hat{f} (k)$.

\begin{proposition}
  Eq.~{\eqref{eq:ns-vort-appr}} has a unique strong solution $\omega^m \in C
  (\mathbb{R}_+, H^{- 2 -} (\mathbb{T}^2))$ for every deterministic initial
  condition in $H^{- 2 -} (\mathbb{T}^2)$. The solution is a strong Markov
  process and it is invariant under~$\mu$.
\end{proposition}

\begin{proof}
  We can rewrite $\omega^m$ in Fourier variables as $\omega^m =
  w_{\mathrm{\tmop{fin}}}^m + w_{\mathrm{\tmop{lin}}}^m \assign \Pi_m \omega^m
  + (1 - \Pi_m) \omega^m$, in such a way that $w_{\mathrm{\tmop{fin}}}^m$ and
  $w_{\mathrm{\tmop{lin}}}^m$ solve a finite-dimensional SDE with locally
  Lipschitz continuous coefficients and an infinite-dimensional linear SDE,
  respectively. Global existence and invariance of $\mu$ follow by Section~7
  in~{\cite{gubinelli2013regularization}}. Now, $w_{\mathrm{\tmop{fin}}}^m$
  has compact spectral support and therefore $w_{\mathrm{\tmop{fin}}}^m \in C
  (\mathbb{R}_+, C^{\infty} (\mathbb{T}))$, while it can be proved that
  $w_{\mathrm{\tmop{lin}}}^m$ has trajectories in $C (\mathbb{R}_+, H^{- 2 -}
  (\mathbb{T}^2))$. Thus, $\omega^m$ has trajectories in $C (\mathbb{R}_+,
  H^{- 2 -} (\mathbb{T}^2))$.
\end{proof}

We define the semigroup of $\omega^m$ for all bounded and measurable functions
$\varphi$ as $T^m_t \varphi (\omega_0) \assign \mathbb{E}_{\omega_0} [\varphi
(\omega_t^m)]$, where, under $\mathbb{P}_{\omega_0}$, the process $\omega^m$
solves {\eqref{eq:ns-vort-appr}} with initial condition~$\omega_0 \in H^{- 2
-} (\mathbb{T}^2)$.

\begin{lemma}
  \label{lem:semigroup}For all $p \in [1, \infty]$, the family of operators
  $(T^m_t)_{t \geqslant 0}$ can be uniquely extended to a contraction
  semigroup on $L^p (\mu)$ which is continuous for $p \in [1, \infty [$.
\end{lemma}

\begin{definition}
  \label{def:cylfn}Let $\mathcal{C}= \tmop{Cyl}_{\mathbb{T}^2}$ denote the set
  of cylinder functions on $H^{- 2 -} (\mathbb{T}^2)$, namely those functions
  $\varphi : H^{- 2 -} (\mathbb{T}^2) \rightarrow \mathbb{R}$ of the form
  $\varphi (\omega) = \Phi (\omega (f_1), \ldots, \omega (f_n))$ for some $n
  \geqslant 1$ where $\Phi : \mathbb{R}^n \rightarrow \mathbb{R}$ is smooth
  and $f_1, \ldots, f_n \in C^{\infty} (\mathbb{T}^2)$, here and in the rest
  of the paper we will denote by $\nu (g)$ the duality between a distribution
  $\nu$ and a function~$g$.
\end{definition}

On such cylinder functions the generator of the semigroup $T^m$ has an
explicit representation: It{\^o}'s formula gives, for $\varphi \in
\tmop{Cyl}_{\mathbb{T}^2}$ as in Definition~\ref{def:cylfn},
\begin{equation}
  \mathd \varphi (\omega_t^m) =\mathcal{L}^m \varphi (\omega^m_t) \mathd t +
  \sum_{i = 1}^n \partial_i \Phi (\omega_t^m (f_1), \ldots, \omega_t^m (f_n))
  \mathd M_t (f_i), \label{eq:ito-cyl}
\end{equation}
where $\mathcal{L}^m \assign \mathcal{L}_{\theta} +\mathcal{G}^m$ with
\[ \mathcal{L}_{\theta} \varphi (\omega) = \sum_{i = 1}^n \partial_i \Phi
   (\omega (f_1), \ldots, \omega (f_n)) \omega (- A^{\theta} f_i) +
   \frac{1}{2} \sum_{i = 1}^n \partial_{i, j}^2 \Phi (\omega (f_1), \ldots,
   \omega (f_n)) \langle A^{\theta + 1} f_i, f_j \rangle, \]
and
\[ \mathcal{G}^m \varphi (\omega) = - \sum_{i = 1}^n \partial_i \Phi (\omega
   (f_1), \ldots, \omega (f_n)) \langle B_m (\omega), f_i \rangle . \]
Here, $(M_t (f_i))_{t \geqslant 0}$ is a continuous martingale with quadratic
variation
\[ \langle M (f_i) \rangle_t = 2 t \| A^{(\theta + 1) / 2} f_i \|_{L^2
   (\mathbb{T}^2)}^2, \]
and therefore $\int_0^{\cdot} \sum_{i = 1}^n \partial_i \Phi (\omega_t^m
(f_1), \ldots, \omega_t^m (f_n)) \mathd M_t (f_i)$ is a martingale.
Consequently, we have
\[ T^m_t \varphi (\omega) - \varphi (\omega) = \int_0^t T^m_s (\mathcal{L}^m
   \varphi) (\omega) \mathd s, \qquad \text{for all } \omega \in H^{- 2 -} .
\]
To extend this to more general functions $\varphi$, we work via Fock space
techniques. The Hilbert space $L^2 (\mu)$ can be identified with the Fock
space $\mathcal{H}= \Gamma H^1_0 (\mathbb{T}^2) \assign \bigoplus_{n =
0}^{\infty} (H^1_0 (\mathbb{T}^2))^{\otimes n}$ with $H^1_0 (\mathbb{T}^2)
\assign \{ \psi \in H^1 (\mathbb{T}^2) : \hat{\psi} (0) = 0 \}$ and norm
\[ \| \varphi \|^2 = \sum_{n = 0}^{\infty} n! \| \varphi_n \|^2_{(H^1_0
   (\mathbb{T}^2))^{\otimes n}} = \sum_{n = 0}^{\infty} n! \sum_{k_{1 : n} \in
   (\mathbb{Z}^2_0)^n} \left( \prod_{i = 1}^n | 2 \pi k_i |^2  \right)  |
   \hat{\varphi}_n (k_{1 : n}) |^2, \]
by noting that any $\varphi \in L^2 (\mu)$ can be written in chaos expansion
$\varphi = \sum_{n \geq 0} W_n (\varphi_n)$, where $W_n$ is the $n$-th order
Wiener-It{\^o} integral and $\varphi_n \in H^1_0 (\mathbb{T}^2)^{\otimes n}$
for every $n \in \mathbb{N}$, see
e.g.~{\cite{janson1997gaussian,nualart2006malliavin}} for details. We will use
the convention that $\varphi_n$ is symmetric in its $n$ arguments, that is, we
identify it with its symmetrisation.\quad Note that cylinder functions are
dense in~$\mathcal{H}$. We denote by $\mathcal{N}$ the \tmtextit{number
operator}, i.e.~the self--adjoint operator on $\mathcal{H}$ such that
$(\mathcal{N} \varphi)_n \assign n \varphi_n$. It is well known that the
semigroup generated by the number operator satisfies an hypercontractivity
estimate, see Theorem~1.4.1 in~{\cite{nualart2006malliavin}}. We record it in
the next lemma.

\begin{lemma}
  \label{lemma:hyper-n}For $p \geqslant 2$, let $c_p = \sqrt{p - 1}$. Then
  \[ \| | \varphi |^{p / 2} \|^2 \leqslant \| c_p^{\mathcal{N}} \varphi \|^p,
     \qquad \text{for every } \varphi \in \mathcal{H}. \]
\end{lemma}

With these preparations we are ready to give expressions for the operators
$\mathcal{L}_{\theta}$ and $\mathcal{G}^m$ in terms of the Fock space
representation of~$\mathcal{H}$.

\begin{lemma}
  For sufficiently nice $\varphi \in \mathcal{H}$, the operator
  $\mathcal{L}_{\theta}$ is given~by
  \begin{equation}
    \mathcal{F} (\mathcal{L}_{\theta} \varphi)_n (k_{1 : n}) = - (2 \pi)^{2
    \theta} L_{\theta} (k_{1 : n}) \hat{\varphi}_n (k_{1 : n})
    \label{eq:ltheta-fourier}
  \end{equation}
  where $L_{\theta} (k_{1 : n}) \assign | k_1 |^{2 \theta} + \cdots + | k_n
  |^{2 \theta} .$ Moreover, writing $\mathcal{G}^m =\mathcal{G}_+^m
  +\mathcal{G}_-^m$ we have
  \begin{eqnarray}
    \mathcal{F} (\mathcal{G}_+^m \varphi)_n (k_{1 : n}) & = & (n - 1)
    \mathbbm{1}_{| k_1 |, | k_2 |, | k_1 + k_2 | \leqslant m}
    \frac{(k_1^{\perp} \cdot (k_1 + k_2)) ((k_1 + k_2) \cdot k_2)}{| k_1 |^2 |
    k_2 |^2} \hat{\varphi}_{n - 1} (k_1 + k_2, k_{3 : n}), 
    \label{eq:gp-fourier}\\
    \mathcal{F} (\mathcal{G}_-^m \varphi)_n (k_{1 : n}) & = & (2 \pi)^2 (n +
    1) n \sum_{p + q = k_1} \mathbbm{1}_{| k_1 |, | p |, | q | \leqslant m}
    \frac{(k_1^{\perp} \cdot p) (k_1 \cdot q)}{| k_1 |^2} \hat{\varphi}_{n +
    1} (p, q, k_{2 : n}) .  \label{eq:gm-fourier}
  \end{eqnarray}
  For all $\varphi_{n + 1} \in (H^1_0 (\mathbb{T}^2))^{\otimes (n + 1)}$ and
  for all $\varphi_n \in (H^1_0 (\mathbb{T}^2))^{\otimes n}$, we have
  \begin{equation}
    \langle \varphi_{n + 1}, \mathcal{G}_+^m \varphi_n \rangle = - \langle
    \mathcal{G}_-^m \varphi_{n + 1}, \varphi_n \rangle . \label{eq:adj}
  \end{equation}
\end{lemma}

\begin{proof}
  The computations are analogous to those of Lemma~3.7
  of~{\cite{gubinelli2018energy}} for $\mathcal{L}_{\theta}$ and of Lemma~2.4
  and Lemma~2.7 in~{\cite{gubinelli2018infinitesimal}}.
\end{proof}

\begin{remark}
  $\mathcal{G}_+^m$ and $\mathcal{G}_-^m$ are (unbounded) operators which
  increase and decrease, respectively, the ``number of particles'' by one.
  Moreover, we know from~{\eqref{eq:adj}} that they are formally the adjoint
  of the other (modulo a sign change).
\end{remark}

A key result is given by the following bounds for $\mathcal{G}_{\pm}^m$ acting
on weighted subspaces of~$\mathcal{H}$.

\begin{lemma}
  \label{lem:uniform-bds-G}Let $w : \mathbb{N}_0 \rightarrow \mathbb{R}_+$ and
  $\varphi \in \mathcal{H}$. The following $m$-dependent bound holds:
  \begin{equation}
    \| w (\mathcal{N}) \mathcal{G}^m \varphi \| \lesssim m \| (w (\mathcal{N}+
    1) + w (\mathcal{N}- 1)) (1 +\mathcal{N}) (1 -\mathcal{L}_{\theta})^{1 /
    2} \varphi \| . \label{eq:m-dep-G}
  \end{equation}
  Moreover, uniformly in~$m$, we have
  \begin{equation}
    \| w (\mathcal{N}) (1 -\mathcal{L}_{\theta})^{- \gamma} \mathcal{G}_+^m
    \varphi \| \lesssim \| w (\mathcal{N}+ 1) (1 +\mathcal{N}) (1
    -\mathcal{L}_{\theta})^{(1 + 1 / \theta) / 2 - \gamma} \varphi \|, \qquad
    \text{for all } \gamma > \frac{1}{2 \theta}, \label{eq:unifbdG1}
  \end{equation}
  and
  \begin{equation}
    \| w (\mathcal{N}) (1 -\mathcal{L}_{\theta})^{- \gamma} \mathcal{G}_-^m
    \varphi \| \lesssim \| w (\mathcal{N}- 1) \mathcal{N}^{3 / 2} (1
    -\mathcal{L}_{\theta})^{(1 + 1 / \theta) / 2 - \gamma} \varphi \|, \qquad
    \text{for all } \gamma < \frac{1}{2} . \label{eq:unifbdG2}
  \end{equation}
\end{lemma}

These bounds will be proven later on in Section~\ref{sec:bounds}. In view of
eq.~{\eqref{eq:m-dep-G}}, it is natural to identify a dense domain
$\mathcal{D} (\mathcal{L}^m)$ for $\mathcal{L}^m$ as
\[ \mathcal{D} (\mathcal{L}^m) \assign \{ \varphi \in \mathcal{H}: \| (1
   +\mathcal{N}) (1 -\mathcal{L}_{\theta}) \varphi \| < \infty \} = (1
   +\mathcal{N})^{- 1} (1 -\mathcal{L}_{\theta})^{- 1} \mathcal{H}. \]
Note that $\langle \psi, (\mathcal{L}_{\theta} +\mathcal{G}^m) \varphi \rangle
= \langle (\mathcal{L}_{\theta} -\mathcal{G}^m) \psi, \varphi \rangle$ for
$\psi, \varphi \in \mathcal{D} (\mathcal{L}^m)$ and in particular that
$\mathcal{L}_{\theta}$ is dissipative since for all $\varphi \in \mathcal{D}
(\mathcal{L}^m)$ we have
\[ \langle \varphi, (\mathcal{L}_{\theta} +\mathcal{G}^m) \varphi \rangle =
   \langle \mathcal{L}_{\theta} \varphi, \varphi \rangle = - \|
   (-\mathcal{L}_{\theta})^{1 / 2} \varphi \|^2 \leqslant 0. \]

A priori $\mathcal{L}^m$ is only the restriction to $\mathcal{D}
(\mathcal{L}^m)$ of the generator $\hat{\mathcal{L}}^m$ of the semigroup
$(T^m_t)_t$. However, we will also prove in Lemma~\ref{eq:lemma-lm} below that
the operator $\mathcal{L}^m$ is closable and that its closure is indeed the
generator~$\hat{\mathcal{L}}^m$.

\

In order to exploit these pieces of information, we have to work with
solutions of Galerkin approximations having ``near-stationary'' fixed-time
marginal.

\begin{definition}
  We say that a stochastic process $(\omega_t)_{t \geqslant 0}$ with values in
  $\mathcal{S}' (\mathbb{T}^2)$ is ($L^2$)-incompressible if, for all $T > 0$,
  there exists a constant $C (T)$ such that we have
  \[ \sup_{0 \leqslant t \leqslant T} \mathbb{E} | \varphi (\omega_t) |
     \leqslant C (T) \| \varphi \|, \qquad \varphi \in \mathcal{C}. \]
\end{definition}

For an incompressible process $(\omega_t)_{t \geqslant 0}$ it makes sense,
using a density argument involving cylinder functions, to define $s \mapsto
\varphi (\omega_s)$ for all $\varphi \in \mathcal{H}$ as a stochastic process
continuous in~$L^1$.

\begin{lemma}
  \label{lem:est-eta}Let $\mathbb{E}_{\eta \mathd \mu}$ be the law of the
  solution $\omega^m$ to the Galerkin approximation~{\eqref{eq:ns-vort-appr}}
  starting from an initial condition $\omega^m_0 \sim \eta \mathd \mu$ with
  $\eta \in L^2 (\mu)$. Then, for any $\Psi : C (\mathbb{R}_+ ; \mathcal{S}')
  \rightarrow \mathbb{R}$,
  \[ \mathbb{E}_{\eta \mathd \mu} | \Psi (\omega^m) | \leqslant \| \eta \|
     \mathbb{E}_{\mu} (\Psi (\omega^m)^2)^{1 / 2} . \]
  In particular, any such process is incompressible uniformly in~$m$.
\end{lemma}

\begin{proof}
  We get
  \[ \mathbb{E}_{\eta \mathd \mu} | \Psi (\omega^m) | =\mathbb{E}_{\mu} [\eta
     (w_0) | \Psi (\omega^m) |] \leqslant \| \eta \| (\mathbb{E}_{\mu} \Psi
     (\omega^m)^2)^{1 / 2} . \]
  Incompressibility easily follows from the fact that $\mu$ is an invariant
  measure for the Galerkin approximations independently of~$m$.
\end{proof}

\begin{definition}
  A {\tmem{weight}} is a measurable increasing map $w : \mathbb{R}_+
  \rightarrow (0, \infty)$ such that there exists $C > 0$ with $w (x)
  \leqslant C w (x + y)$, for all $x \geqslant 1$ and for $| y | \leqslant 1$.
  We write as $| w |$ the smallest such constant~$C$. We denote $w
  (\mathcal{N})$ the self-adjoint operator on $\mathcal{H}$ defined as
  spectral multiplier.
\end{definition}

We will use the notation $D_x$ to indicate the Malliavin derivative, see
e.g.~{\cite{gubinelli2018infinitesimal}}, which acts on cylinder functions
$\varphi$ as in Definition~\ref{def:cylfn} as follows,
\[ D_x \varphi = \sum_{k = 1}^n \partial_{x_k} \Phi (\omega (f_1), \ldots,
   \omega (f_n)) f_k (x) . \]
\begin{lemma}
  \label{lemma:galerkin-estimates}Let $\eta \in L^2 (\mu)$ and let $\omega^m$
  be a solution to~{\eqref{eq:ns-vort-appr}} with $\tmop{Law} (\omega^m_0)
  \sim \eta \mathd \mu$. Then this solution is incompressible and, for any
  $\varphi \in \mathcal{D} (\mathcal{L}^m)$, the process
  \[ M_t^{m, \varphi} = \varphi (\omega^m_t) - \varphi (\omega^m_0) - \int_0^t
     \mathcal{L}^m \varphi (\omega_s^m) \mathd s, \qquad t \geqslant 0,
     \label{eq:mart-prob-m} \]
  is a continuous martingale with quadratic variation
  \begin{equation}
    \langle M^{m, \varphi} \rangle_t = \int_0^t \mathcal{E} (\varphi)
    (\omega^m_s) \mathd s, \qquad \text{with} \qquad \mathcal{E} (\varphi) = 2
    \int_{\mathbb{T}^2} |A^{\frac{\theta + 1}{2}}_x D_x \varphi |^2 \mathd x.
    \label{eq:mart-qvar}
  \end{equation}
  For any weight $w$, we have
  \begin{equation}
    \| w (\mathcal{N}) (\mathcal{E} (\varphi))^{1 / 2} \| \lesssim \| w
    (\mathcal{N}- 1) (1 -\mathcal{L}_{\theta})^{1 / 2} \varphi \| .
    \label{eq:energy-norm}
  \end{equation}
  Moreover, for all $p \geqslant 1$, it holds
  \begin{equation}
    \mathbb{E} \sup_{t \in [0, T]} \left| \int_0^t \varphi (\omega_s^m) \mathd
    s \right|^p \lesssim (T^{p / 2} \vee T^p) \| c_{2 p}^{\mathcal{N}} (1
    -\mathcal{L}_{\theta})^{- 1 / 2} \varphi \|^p, \label{eq:ito-trick-m}
  \end{equation}
  uniformly in~$m$.
\end{lemma}

\begin{proof}
  If $\varphi$ is a cylinder function, then we have eq.~{\eqref{eq:ito-cyl}}
  and in that case Doob's inequality and Lemma~\ref{lem:est-eta} yield, for
  all $T > 0$,
  \[ \mathbb{E} \sup_{t \in [0, T]} | M_t^{m, \varphi} | \lesssim \mathbb{E}
     (\langle M^{m, \varphi} \rangle_T^{1 / 2}) \leqslant \| \eta \|
     \mathbb{E}_{\mu} (\langle M^{m, \varphi} \rangle_T)^{1 / 2} \lesssim \|
     \eta \| T^{1 / 2} \| (\mathcal{E} (\varphi))^{1 / 2} \|_{L^2 (\mu)} . \]
  The norm appearing on the right-hand side can be estimated as follows:
  \begin{eqnarray*}
    \| w (\mathcal{N}) (\mathcal{E} (\varphi))^{1 / 2} \|^2 & = & 2 \int_x
    \left\| w (\mathcal{N}) A^{\frac{\theta + 1}{2}}_x D_x \varphi \right\|^2
    \mathd x\\
    & = & 2 \int_x \left( \sum_{n \geq 0} (n - 1) !w (n - 1)^2 n^2 \left\|
    A^{\frac{\theta + 1}{2}}_x \varphi_n (x, \cdot) \right\|^2_{H^1_0
    (\mathbb{T}^2)^{\otimes (n - 1)}} \right)\\
    & \simeq & 2 \sum_{n \geq 1} n!w (n - 1)^2 n \sum_{k_{1 : n}} \left(
    \prod_{i = 2}^n | 2 \pi k_i |^2 \right) | k_1 |^{2 (\theta + 1)} |
    \hat{\varphi}_n (k_{1 : n}) |^2\\
    & = & 2 \sum_{n \geq 1} n!w (n - 1)^2 n \sum_{k_{1 : n}} \left( \prod_{i
    = 1}^n | 2 \pi k_i |^2 \right) | k_1 |^{2 \theta} | \hat{\varphi}_n (k_{1
    : n}) |^2\\
    & = & 2 \sum_{n \geq 1} n!w (n - 1)^2 \sum_{k_{1 : n}} \left( \prod_{i =
    1}^n | 2 \pi k_i |^2 \right) L_{\theta} (k_{1 : n}) | \hat{\varphi}_n
    (k_{1 : n}) |^2\\
    & \lesssim & 2 \| w (\mathcal{N}- 1) (1 -\mathcal{L}_{\theta})^{1 / 2}
    \varphi \|^2,
  \end{eqnarray*}
  where we used a symmetrisation in the arguments of $\hat{\varphi}_n$ in the
  5th line. Using the bounds~{\eqref{eq:m-dep-G}}
  and~{\eqref{eq:energy-norm}}, one can extend
  formula~{\eqref{eq:mart-prob-m}} to all functions in $\mathcal{D}
  (\mathcal{L}^m)$ by a density argument.
  
  \
  
  As far as~{\eqref{eq:ito-trick-m}} is concerned, let us remark that,
  provided the process $\omega^m$ is started from its stationary measure
  $\mu$, then the \tmtextit{reversed process} $(\tilde{\omega}_t = \omega_{T -
  t})_{t \geqslant 0}$ is also stationary and with (martingale) generator
  $\tilde{\mathcal{L}}^m =\mathcal{L}_{\theta} -\mathcal{G}^m$. The
  forward--backward It{\^o} trick used in~{\cite{gubinelli2013regularization}}
  allows us to represent additive functionals of the form $\int_0^t
  \mathcal{L}_{\theta} \psi (\omega_s^m) \mathd s$ as a sum of forward and
  backward martingales whose quadratic variations
  satisfy~{\eqref{eq:mart-qvar}}. Therefore,
  \begin{equation}
    \begin{array}{l}
      \mathbb{E}_{\mu} \left[ \sup_{t \in [0, T]} \left| \int_0^t
      \mathcal{L}_{\theta} \varphi (\omega_s^m) \mathd s \right|^p \right]
      \lesssim T^{p / 2} \| (\mathcal{E} (\varphi))^{p / 4} \|^2\\
      \qquad \lesssim T^{p / 2} \| c_p^{\mathcal{N}} \mathcal{E} (\varphi)^{1
      / 2} \|^p \lesssim T^{p / 2} \| c_p^{\mathcal{N}} (1
      -\mathcal{L}_{\theta})^{1 / 2} \varphi \|^p .
    \end{array} \label{eq:ito-trick2}
  \end{equation}
  Let $\psi = (1 -\mathcal{L}_{\theta})^{- 1} \varphi$ and
  exploit~{\eqref{eq:ito-trick2}} to compute
  \[ \begin{array}{lll}
       \mathbb{E}_{\mu} \left[ \sup_{t \in [0, T]} \left| \int_0^t \varphi
       (\omega_s^m) \mathd s \right|^p \right] & = & \mathbb{E}_{\mu} \left[
       \sup_{t \in [0, T]} \left| \int_0^t (1 -\mathcal{L}_{\theta}) \psi
       (\omega_s^m) \mathd s \right|^p \right]\\
       & \lesssim & \mathbb{E}_{\mu} \left[ \sup_{t \in [0, T]} \left|
       \int_0^t (-\mathcal{L}_{\theta}) \psi (\omega_s^m) \mathd s \right|^p
       \right] +\mathbb{E}_{\mu} \left[ \sup_{t \in [0, T]} \left| \int_0^t
       \psi (\omega_s^m) \mathd s \right|^p \right]\\
       & \lesssim & T^{p / 2} \| c_p^{\mathcal{N}} (1
       -\mathcal{L}_{\theta})^{1 / 2} \psi \|^p + T^p \| c_p^{\mathcal{N}}
       \psi \|^p\\
       & \lesssim & (T^{p / 2} \vee T^p) (\| c_p^{\mathcal{N}} (1
       -\mathcal{L}_{\theta})^{- 1 / 2} \varphi \|^p + \| c_p^{\mathcal{N}} (1
       -\mathcal{L}_{\theta})^{- 1} \varphi \|^p)\\
       & \lesssim & (T^{p / 2} \vee T^p) \| c_p^{\mathcal{N}} (1
       -\mathcal{L}_{\theta})^{- 1 / 2} \varphi \|^p,
     \end{array} \]
  which is uniform in~$m$.
\end{proof}

\section{The cylinder martingale problem}\label{sec:cyl-mart}

We want now to take limits of Galerkin approximations and have a
characterisation of the limiting dynamics. The main problem is that the formal
limiting (martingale) generator $\mathcal{L}$ does not send cylinder functions
to $\mathcal{H}$, therefore we cannot properly formulate a martingale problem
for incompressible solutions. However, estimate~{\eqref{eq:ito-trick-m}}
suggests that it is reasonable to ask that any limit process $(\omega_t)_{t
\geqslant 0}$ satisfies
\begin{equation}
  \mathbb{E} \sup_{t \in [0, T]} \left| \int_0^t \varphi (\omega_s) \mathd s
  \right|^p \lesssim (T^{p / 2} \vee T^p) \| c_{2 p}^{\mathcal{N}} (1
  -\mathcal{L}_{\theta})^{- 1 / 2} \varphi \|^p, \label{eq:ito-trick}
\end{equation}
for all $p \geqslant 1$ and all cylinder functions $\varphi \in \mathcal{C}$.
The proof of the next lemma is almost immediate.

\begin{lemma}
  \label{lemma:integral}Assume that a process $(\omega_t)_t$
  satisfies~{\eqref{eq:ito-trick}} and let $I_t (\varphi) = \int_0^t \varphi
  (\omega_s) \mathd s$ for all $\varphi \in \mathcal{C}$. Then the map
  $\varphi \mapsto (I_t (\varphi))_{t \geqslant 0}$ can be extended to all
  $\varphi \in (1 -\mathcal{L}_{\theta})^{1 / 2} \mathcal{H}$. The process
  $(I_t (\varphi))_{t \geqslant 0}$ is almost surely continuous.
\end{lemma}

\begin{proof}
  Take $(\varphi_n)_n \subseteq \mathcal{C}$ such that $\sum_n \| (1
  -\mathcal{L}_{\theta})^{1 / 2} \varphi_n - \varphi \| < \infty$, then it is
  easy to see that $(I (\varphi_n))_n$ is a Cauchy sequence in $C ([0, T] ;
  \mathbb{R})$ a.s. with limit $I (\varphi) \in C ([0, T] ; \mathbb{R})$. It
  satisfies~{\eqref{eq:ito-trick}} by Fatou's lemma and, therefore, depends
  only on $\varphi$ and not on the particular approximating sequence.
\end{proof}

From this we deduce that for such processes we have
\[ \lim_{m \rightarrow \infty} \int_0^t (\mathcal{L}^m \varphi) (\omega_s)
   \mathd s = \int_0^t (\mathcal{L} \varphi) (\omega_s) \mathd s, \]
in probability and in $L^p$ for cylinder functions $\varphi \in \mathcal{C}$.
Here, on the right-hand side the quantity $\mathcal{L} \varphi$ is defined as
$\mathcal{L} \varphi =\mathcal{L}_{\theta} \varphi + \lim_{m \rightarrow
\infty} \mathcal{G}^m \varphi$, that is an element of the space of
distributions $(1 -\mathcal{L}_{\theta})^{1 / 2} \mathcal{H}$. The limit
exists and is unique thanks to the uniform estimates on $\mathcal{G}^m$ in
Lemma~\ref{lem:uniform-bds-G}. As a consequence, we have also a notion of
martingale problem w.r.t.~the operator $\mathcal{L}$ involving only cylinder
functions.

\begin{definition}
  \label{def:martpb-cyl}A process $(\omega_t)_{t \geqslant 0}$ with
  trajectories in $C (\mathbb{R}_+ ; \mathcal{S}')$ solves the {\tmem{cylinder
  martingale problem for $\mathcal{L}$ with initial distribution $\nu$}} if
  $\omega_0 \sim \nu$ and if the following conditions are satisfied:
  \begin{enumerateroman}
    \item $(\omega_t)_t$ is incompressible,
    
    \item the It{\^o} trick works: for all cylinder functions $\varphi$ and
    all $p \geqslant 1$, we have eq.~{\eqref{eq:ito-trick}}.
    
    \item for any $\varphi \in \mathcal{C}$, the process
    \begin{equation}
      M_t^{\varphi} = \varphi (\omega_t) - \varphi (\omega_0) - \int_0^t
      \mathcal{L} \varphi (\omega_s) \mathd s, \qquad t \geqslant 0,
      \label{eq:mart-prob-cyl}
    \end{equation}
    is a continuous martingale with quadratic variation $\langle M^{\varphi}
    \rangle_t = \int_0^t \mathcal{E} (\varphi) (\omega_s) \mathd s$. The
    integral on the right-hand side of eq.~{\eqref{eq:mart-prob-cyl}} is
    defined according to Lemma~\ref{lemma:integral}.
  \end{enumerateroman}
\end{definition}

\begin{theorem}
  \label{thm:existence-cyl-martingalepb}Let $\eta \in L^2 (\mu)$ and, for each
  $m \geqslant 1$, let $(\omega^m)$ be the solution
  to~{\eqref{eq:ns-vort-appr}} with $\omega^m_0 \sim \eta \mathd \mu$. Then
  the family $(\omega^m)_{m \in \mathbb{N}}$ is tight in $C (\mathbb{R}_+ ;
  \mathcal{S}')$ and any weak limit $\omega$ solves the cylinder martingale
  problem for $\mathcal{L}$ with initial distribution~$\eta \mathd \mu$
  according to Definition~\ref{def:martpb-cyl} and we have
  \begin{equation}
    \mathbb{E} [| \varphi (\omega_t) - \varphi (\omega_s) |^p] \lesssim (| t -
    s |^{p / 2} \vee | t - s |^p) \| c_{4 p}^{\mathcal{N}} (1
    -\mathcal{L}_{\theta})^{- 1 / 2} \varphi \|^p \label{eq:mod-cont}
  \end{equation}
  for any $p \geqslant 2$ and $\varphi \in \mathcal{C}$.
\end{theorem}

\begin{proof}
  The proof follows the one for Theorem~4.6
  in~{\cite{gubinelli2018infinitesimal}}.
  
  \
  
  \tmtextbf{Step 1.} Consider $p \geqslant 2$ and $\varphi \in \mathcal{C}$.
  We want to derive an estimate for $\mathbb{E} [| \varphi (\omega^m_t) -
  \varphi (\omega_s^m) |^p]$. We write then $\varphi (\omega^m_t) - \varphi
  (\omega_s^m) = \int_s^t \mathcal{L}^m \varphi (\omega^m_r) \mathd r +
  M_t^{m, \varphi} - M_s^{m, \varphi}$, and get from Lemma~\ref{lem:est-eta}
  and eq.~{\eqref{eq:ito-trick-m}} the following bound
  \[ \begin{array}{lll}
       \mathbb{E} \left[ \left| \int_s^t \mathcal{L}^m \varphi (\omega^m_r)
       \mathd r \right|^p \right] & \lesssim & \left[ \mathbb{E}_{\mu} \left|
       \int_s^t \mathcal{L}^m \varphi (\omega_r^m) \mathd r \right|^{2 p}
       \right]^{1 / 2}\\
       & \lesssim & (| t - s |^{p / 2} \vee | t - s |^p) \| c_{4
       p}^{\mathcal{N}} (1 -\mathcal{L}_{\theta})^{- 1 / 2} \varphi \|^p .
     \end{array} \]
  The martingale term can be bounded by means of the Burkholder-Davis-Gundy
  inequality and~{\eqref{eq:energy-norm}} as follows:
  \[ \begin{array}{lll}
       \mathbb{E} [| M_t^{m, \varphi} - M_s^{m, \varphi} |^p] & \lesssim &
       \mathbb{E} \left[ \left( \int_s^t \mathcal{E} (\varphi) (\omega^m_r)
       \mathd r \right)^{p / 2} \right] \lesssim \left[ \mathbb{E}_{\mu}
       \left( \int_s^t \mathcal{E} (\varphi) (\omega_r^m) \mathd r \right)^p
       \right]^{1 / 2}\\
       & \lesssim & | t - s |^{p / 2} \| (\mathcal{E} (\varphi))^{p / 2} \|
       \lesssim | t - s |^{p / 2} \| c_{2 p}^{\mathcal{N}} (\mathcal{E}
       (\varphi))^{1 / 2} \|^p\\
       & \lesssim & | t - s |^{p / 2} \| c_{2 p}^{\mathcal{N}} (1
       -\mathcal{L}_{\theta})^{1 / 2} \varphi \|^p .
     \end{array} \]
  Therefore,
  \begin{equation}
    \mathbb{E} [| \varphi (\omega_t^m) - \varphi (\omega_s^m) |^p] \lesssim (|
    t - s |^{p / 2} \vee | t - s |^p) \| c_{4 p}^{\mathcal{N}} (1
    -\mathcal{L}_{\theta})^{- 1 / 2} \varphi \|^p . \label{eq:bd-tightness}
  \end{equation}
  The law of the initial condition $\varphi (\omega_0^m)$ is independent
  of~$m$, and by Kolmogorov's continuity criterion the sequence of real-valued
  processes $(\varphi (\omega^m))_m$ is tight in $C (\mathbb{R}_+ ;
  \mathbb{R})$ whenever $p \geqslant 4$ and $\varphi \in \mathcal{C}$ is such
  that $\| c_{4 p}^{\mathcal{N}} (1 -\mathcal{L}_{\theta})^{- 1 / 2} \varphi
  \| < \infty$. Note that this space contains in particular all the functions
  of the form $\varphi (\omega) = \omega (f)$ with $f \in C^{\infty}
  (\mathbb{T}^2)$. Hence, we can apply Mitoma's criterion~{\cite{Mitoma1983}}
  to get the tightness of the sequence $(\omega^m)_m$ in $C (\mathbb{R}_+ ;
  \mathcal{S}')$.
  
  \
  
  \tmtextbf{Step 2.} Since $\omega_0^m \sim \eta \mathd \mu$, any weak limit
  has initial distribution~$\eta \mathd \mu$. Incompressibility is also clear
  since, for any $\varphi \in \mathcal{H}$, we have
  \[ \mathbb{E} [| \varphi (\omega_t) |] \leqslant \underset{m \rightarrow
     \infty}{\lim \inf} \mathbb{E} [| \varphi (\omega_t^m) |] \leqslant \|
     \eta \| \| \varphi \| . \]
  Using cylinder functions, we can pass to the limit in
  eq.~{\eqref{eq:ito-trick-m}} and prove that any accumulation point
  $(\omega_t)_t$ satisfies eq.~{\eqref{eq:ito-trick}}. It remains to check the
  martingale characterisation~{\eqref{eq:mart-prob-cyl}}. Fix $\varphi \in
  \mathcal{C}$ and let $(\psi_n)_n \subseteq \mathcal{C}$ be such that $\psi_n
  \rightarrow \mathcal{L} \varphi$ in $(1 +\mathcal{L}_{\theta})^{1 / 2}
  \mathcal{H}$. By convergence in law, incompressibility,
  eq.~{\eqref{eq:ito-trick-m}} and eq.~{\eqref{eq:ito-trick}}, we have that
  \[ \begin{array}{ll}
       & \mathbb{E} \left[ \left( \varphi (\omega_t) - \varphi (\omega_s) -
       \int_s^t \mathcal{L} \varphi (\omega_r) \mathd r \right) G
       ((\omega_r)_{r \in [0, s]}) \right]\\
       = & \lim_{n \rightarrow \infty} \mathbb{E} \left[ \left( \varphi
       (\omega_t) - \varphi (\omega_s) - \int_s^t \psi_n (\omega_r) \mathd r
       \right) G ((\omega_r)_{r \in [0, s]}) \right]\\
       = & \lim_{n \rightarrow \infty} \lim_{m \rightarrow \infty} \mathbb{E}
       \left[ \left( \varphi (\omega_t^m) - \varphi (\omega_s^m) - \int_s^t
       \psi_n (\omega_r^m) \mathd r \right) G ((\omega_r^m)_{r \in [0, s]})
       \right]\\
       = & \lim_{m \rightarrow \infty} \mathbb{E} \left[ \left( \varphi
       (\omega_t^m) - \varphi (\omega_s^m) - \int_s^t \mathcal{L} \varphi
       (\omega_r^m) \mathd r \right) G ((\omega_r^m)_{r \in [0, s]}) \right],
     \end{array} \]
  where the exchange of limits in the last line is justified by the uniformity
  in $m$ of the bound in eq.~{\eqref{eq:ito-trick-m}}. By dominated
  convergence in the estimates leading to Lemma~\ref{lemma:galerkin-estimates}
  one has
  \[ \| (1 -\mathcal{L}_{\theta})^{- 1 / 2} (\mathcal{L} \varphi
     -\mathcal{L}^m \varphi) \| = \| (1 -\mathcal{L}_{\theta})^{- 1 / 2}
     (\mathcal{G} \varphi -\mathcal{G}^m \varphi) \| \lesssim o (1) \| (1
     +\mathcal{N})^{3 / 2} (1 -\mathcal{L}_{\theta})^{1 / 2} \varphi \| \]
  as $m \rightarrow \infty$. This is enough to conclude (again using
  eq.~{\eqref{eq:ito-trick-m}}) that
  \begin{equation}
    \begin{array}{ll}
      & \lim_{m \rightarrow \infty} \mathbb{E} \left[ \left( \varphi
      (\omega_t^m) - \varphi (\omega_s^m) - \int_s^t \mathcal{L} \varphi
      (\omega_r^m) \mathd r \right) G ((\omega_r^m)_{r \in [0, s]}) \right]\\
      = & \lim_{m \rightarrow \infty} \mathbb{E} \left[ \left( \varphi
      (\omega_t^m) - \varphi (\omega_s^m) - \int_s^t \mathcal{L}^m \varphi
      (\omega_r^m) \mathd r \right) G ((\omega_r^m)_{r \in [0, s]}) \right] =
      0,
    \end{array} \label{eq:limit-existence}
  \end{equation}
  since $(\omega^m_t)_t$ solves indeed the martingale problem
  for~$\mathcal{L}^m$. This establishes that any accumulation point
  $(\omega_t)_t$ is a solution to the cylinder martingale problem
  for~$\mathcal{L}$. Similarly, one can pass to the limit on the martingales
  $(M_t^{m, \varphi})_t$ to show that the limiting quadratic variation is as
  claimed.
\end{proof}

\section{Uniqueness of solutions}\label{sec:uniq}

Uniqueness of solutions to the cylinder martingale problem depends on the
control of the associated Kolmogorov equation.

\

The following standard fact on generators of semigroups that will be useful
in our further considerations. For the sake of the reader we provide also a
proof to illustrate the relation between the Kolmogorov equation for a
concrete operator and abstract semigroup theory.

\begin{lemma}
  \label{lemma:closure}Let $\mathcal{A}$ be a densely defined, dissipative
  operator on $\mathcal{H}$ and assume that \ we can solve the Kolmogorov
  equation $\partial_t \varphi (t) =\mathcal{A} \varphi (t)$ in $C
  (\mathbb{R}_+ ; \mathcal{D} (\mathcal{A})) \cap C^1 (\mathbb{R}_+ ;
  \mathcal{H})$ with initial condition $\varphi (0) = \varphi_0$ in a dense
  set $\mathcal{U}_{\mathcal{A}} \subseteq \mathcal{D} (\mathcal{A})$. Then
  $\mathcal{A}$ is closable and its closure $\mathcal{B}$ is the unique
  extension of $\mathcal{A}$ which generates a strongly continuous semigroup
  of contractions $(T_t)_{t \geqslant 0}$. Moreover, we have
  \begin{equation}
    \mathcal{A}T_t \varphi_0 = T_t \mathcal{A} \varphi_0, \label{eq:kolm-aa}
  \end{equation}
  for all $\varphi_0 \in \mathcal{U}_{\mathcal{A}}$.
\end{lemma}

\begin{proof}
  Since $\mathcal{A}$ is dissipative, the solution to the Kolmogorov equation
  is unique and $\| \varphi (t) \| \leqslant \| \varphi_0 \|$. Then, if we let
  $T_t \varphi_0 = \varphi (t)$ for $\varphi_0 \in \mathcal{U}_{\mathcal{A}}$
  we can extend $T_t$ by continuity to the whole space $\mathcal{H}$ as a
  contraction. By uniqueness, we have then $T_{t + s} \varphi_0 = T_t T_s
  \varphi_0$, since $t \mapsto T_{t + s} \varphi_0$ solves the equation with
  initial condition $T_s \varphi_0$. Moreover, for $\varphi_0 \in
  \mathcal{U}_{\mathcal{A}}$, we have that
  \begin{equation}
    T_t \varphi_0 - \varphi_0 = \int_0^t \mathcal{A}T_s \varphi_0 \mathd s,
    \label{eq:semiga}
  \end{equation}
  which implies that $t \mapsto T_t \varphi_0$ is strongly continuous. Again
  by density, we deduce that $(T_t)_{t \geqslant 0}$ is a strongly continuous
  semigroup. Let now $\mathcal{B}$ be its Hille--Yosida generator.
  Then~{\eqref{eq:semiga}} implies that $\mathcal{B} \varphi_0 = \partial_t
  T_t \varphi_0 |_{t = 0} =\mathcal{A} \varphi_0$ for all $\varphi_0 \in
  \mathcal{U}_{\mathcal{A}}$, and therefore for all $\varphi_0 \in \mathcal{D}
  (\mathcal{A})$ since $\mathcal{B}$ is closed. So $\mathcal{B}$ is an
  extension of $\mathcal{A}$ and therefore $\mathcal{A}$ is closable. Assume
  now that there exists another extension $\tilde{\mathcal{B}}$ which is the
  generator of another strongly continuous semigroup $(S_t)_{t \geqslant 0}$
  of contractions. \ Now, for all $\varphi_0 \in \mathcal{U}_{\mathcal{A}}
  \subseteq \mathcal{D} (\mathcal{A}) \subseteq \mathcal{D}
  (\tilde{\mathcal{B}})$ we have $\partial_t S_t \varphi_0 =
  \tilde{\mathcal{B}} S_t \varphi_0$, but also $\partial_t T_t \varphi_0
  =\mathcal{A}T_t \varphi_0 = \tilde{\mathcal{B}} T_t \varphi_0$. Since
  $\tilde{\mathcal{B}}$ is dissipative (due to the fact that its semigroup is
  contractive), the associated Kolmogorov equation must have a unique solution
  and, as a consequence, $T_t \varphi_0 = S_t \varphi_0$, which by density
  implies that $T = S$ and that $\mathcal{B}= \tilde{\mathcal{B}}$. Now
  observe that, if $\varphi_0 \in \mathcal{U}_{\mathcal{A}}$, then $T_t
  \varphi_0 \in \mathcal{D} (\mathcal{A})$ and by standard results on
  contraction semigroups (see e.g.~Proposition~1.1.5
  in~{\cite{ethier_markov_2005}}) we have $\mathcal{A}T_t \varphi_0
  =\mathcal{B}T_t \varphi_0 = T_t \mathcal{B} \varphi_0 = T_t \mathcal{A}
  \varphi_0$.
\end{proof}

Theorem~\ref{thm:exist-kolmogorov} below tells us that we can find a dense
domain $\mathcal{D} (\mathcal{L}) \subseteq \mathcal{H}$ for $\mathcal{L}$
such that the Kolmogorov equation
\begin{equation}
  \partial_t \varphi (t) =\mathcal{L} \varphi (t), \qquad t \geqslant 0,
  \label{eq:kolm}
\end{equation}
has a unique solution in $C (\mathbb{R}_+ ; \mathcal{D} (\mathcal{L})) \cap
C^1 (\mathbb{R}_+ ; \mathcal{H})$ for any initial condition in a dense set
$\mathcal{U} \subseteq \mathcal{H}$. As a first consequence,
Lemma~\ref{lemma:closure} tells us that $\mathcal{L}$ is closable and its
closure $\mathcal{L}^{\natural}$ is the generator of a strongly continuous
semigroup $(T_t)_{t \geqslant 0}$ and $\varphi (t) = T_t \varphi$ for all
$\varphi \in \mathcal{U}$.

\begin{lemma}
  \label{lem:ismart}Let $\varphi \in C (\mathbb{R}_+ ; \mathcal{D}
  (\mathcal{L})) \cap C^1 (\mathbb{R}_+ ; \mathcal{H})$ and let $\omega$ be a
  solution to the cylinder martingale problem for~$\mathcal{L}$. Then
  \[ \varphi (t, \omega_t) - \varphi (0, \omega_0) - \int_0^t (\partial_s
     +\mathcal{L}) \varphi (s, \omega_s) \mathd s, \qquad t \geqslant 0, \]
  is a martingale.
\end{lemma}

\begin{proof}
  By an approximation argument, it is easy to see that for any $\varphi \in
  \mathcal{D} (\mathcal{L})$ the process
  \[ \varphi (\omega_t) - \varphi (\omega_0) - \int_0^t \mathcal{L} \varphi
     (\omega_s) \mathd s, \qquad t \geqslant 0, \]
  is a martingale, where the integral on the right-hand side is now understood
  as a standard Lebesgue integral of the continuous process $s \mapsto
  (\mathcal{L} \varphi) (\omega_s)$ (which is well defined a.s.). The proof of
  the extension to time-dependent functions follows the same lines as that of
  Lemma~A.3 in~{\cite{gubinelli2018infinitesimal}}.
\end{proof}

For an incompressible process we have that, for all $t \geqslant 0$,
\[ \int_0^s (\partial_r +\mathcal{L}) T_{t - r} \varphi (\omega_r) \mathd r =
   0, \qquad s \in [0, t] \]
for all $\varphi \in \mathcal{D} (\mathcal{L}^{\natural})$, and therefore also
that $(T_{t - s} \varphi (\omega_s))_{s \in [0, t]}$ is a martingale for any
solution of the cylinder martingale problem for~$\mathcal{L}$. This easily
implies the main result of the paper.

\begin{theorem}
  \label{thm:unique-mart}There exists a unique solution $\omega$ to the
  cylinder martingale problem for $\mathcal{L}$ with initial distribution
  $\omega_0 \sim \eta \mathd \mu$ with $\eta \in L^2 (\mu)$. Moreover,
  $\omega$ is a homogeneous Markov process with transition kernel $(T_t)_{t
  \geqslant 0}$ and with invariant measure~$\mu$.
\end{theorem}

\begin{proof}
  Let us first prove that $(\omega_t)_{t \geqslant 0}$ is Markov. Let $0
  \leqslant t < s$, let $X$ be an $\mathcal{F}_t$-measurable bounded random
  variable, where $\mathcal{F}_t = \sigma (\omega_r : r \in [0, t])$, and let
  $\varphi \in \mathcal{D} (\mathcal{L}^{\natural})$, then $(T_{s - t} \varphi
  (\omega_t))_{t \in [0, s]}$ is a martingale and
  \[ \begin{array}{lll}
       \mathbb{E} [X \varphi (\omega_s)] & = & \mathbb{E} [X T_{s - s} \varphi
       (\omega_s)] =\mathbb{E} [X T_{s - t} \varphi (\omega_t)]
     \end{array} \]
  i.e., $\mathbb{E} [\varphi_0 (\omega_s) | \mathcal{F}_t \nobracket] = T_{s -
  t} \varphi (\omega_t) =\mathbb{E} [\varphi_0 (\omega_s) | \omega_t
  \nobracket]$, and the Markov property is a consequence of another density
  argument. Moreover, its transition kernel is given by the semigroup
  $(T_t)_{t \geqslant 0}$. By an induction argument, it is clear that any
  finite-dimensional marginal is determined by $T$ and by the law of $\omega_0
  \sim \eta \mathd \mu$. As a consequence, the law of the process is unique.
  If $\omega_0 \sim \mu$, then the process is stationary.
\end{proof}

\begin{remark}
  As a by-product note that the formula $(T_t \varphi) (\omega_0) =\mathbb{E}
  [\varphi (\omega_t) | \omega_0]$ allows to extend the semigroup $T$ to a
  bounded semigroup in $L^p$ for all $p \in [1, \infty]$ since
  \[ | \langle \psi, T_t \varphi \rangle | = | \mathbb{E}_{\mu} [\psi
     (\omega_0) (T_t \varphi) (\omega_0)] | = | \mathbb{E}_{\mu} [\psi
     (\omega_0) \varphi (\omega_t)] | \leqslant \| \psi \|_{L^p} \| \varphi
     \|_{L^q} \]
  for all $\psi, \varphi \in L^{\infty} (\mu)$ and all $p, q \in [1, \infty]$
  with $1 / p + 1 / q = 1$. Therefore $\| T_t \varphi \|_{L^q} \leqslant \|
  \varphi \|_{L^q}$. Moreover for all $\varphi \in \mathcal{C}$ such that $\|
  c_{4 p}^{\mathcal{N}} (1 -\mathcal{L}_{\theta})^{- 1 / 2} \varphi \| <
  \infty$ we have
  \[ \| T_t \varphi - \varphi \|_{L^p} = \sup_{\psi : \| \psi \|_{L^q}
     \leqslant 1} \mathbb{E}_{\mu} [\psi (\omega_0) (\varphi (\omega_t) -
     \varphi (\omega_0))] \leqslant (\mathbb{E}_{\mu} [| \varphi (\omega_t) -
     \varphi (\omega_0) |^p])^{1 / p} \rightarrow 0 \]
  as $t \rightarrow 0$ by eq.~{\eqref{eq:mod-cont}}. An approximation argument
  gives that $(T_t)_{t \geqslant 0}$ is strongly continuous in $L^p$ for all
  $1 \leqslant p < \infty$.
\end{remark}

\section{The Kolmogorov equation}\label{s:kolmogorov}

It remains to determine a suitable domain for $\mathcal{L}$ and solve the
Kolmogorov backward equation
\[ \partial_t \varphi (t) =\mathcal{L} \varphi (t), \]
for a sufficiently large class of initial data. In order to do so, we consider
the backward equation for the Galerkin approximation with generator
$\mathcal{L}^m$ and derive uniform estimates. By compactness, this yields the
existence of strong solutions to the backward equation after removing the
cutoff. Uniqueness follows by the dissipativity of~$\mathcal{L}$.

\subsection{A priori estimates }\label{s:a-priori-bounds}

\begin{lemma}
  \label{lem:exist-m-kolmogorov}For any $\varphi_0 \in \mathcal{V} \assign (1
  +\mathcal{N})^{- 2} (1 -\mathcal{L}_{\theta})^{- 1} \mathcal{H}$, there
  exists a solution
  \[ \varphi^m \in C (\mathbb{R}_+, \mathcal{D} (\mathcal{L}^m)) \cap C^1
     (\mathbb{R}_+ ; \mathcal{H}) \]
  to the backward Kolmogorov equation
  \[ \partial_t \varphi^m (t) =\mathcal{L}^m \varphi^m (t) \]
  with $\varphi^m (0) = \varphi_0$ and which satisfies the estimates
  \[ \| (1 +\mathcal{N})^p \varphi^m (t) \|^2 + \int_0^t e^{- C (t - s)} \| (1
     +\mathcal{N})^p (1 -\mathcal{L}_{\theta})^{1 / 2} \varphi^m (s) \|^2
     \mathd s \lesssim_p e^{C t} \| (1 +\mathcal{N})^p \varphi_0 \|^2, \]
  and
  \[ \| (1 +\mathcal{N})^p (1 -\mathcal{L}_{\theta}) \varphi^m (t) \|
     \lesssim_{t, m, p} \| (1 +\mathcal{N})^{p + 1} (1 -\mathcal{L}_{\theta})
     \varphi_0 \|, \]
  for all $t \geqslant 0$ and $p \geqslant 1$.
\end{lemma}

\begin{proof}
  Take $h > 0$ and let $\mathcal{G}^{m, h} = J_h \mathcal{G}^m J_h$, where
  $J_h = e^{- h (\mathcal{N}-\mathcal{L}_{\theta})}$. The operator
  $\mathcal{G}^{m, h}$ is bounded on $\mathcal{H}$ by the estimates in
  Lemma~\ref{lem:uniform-bds-G}. Consider $\varphi_0^m \in \mathcal{D}
  (\mathcal{L}^m)$. Using the fact that $\mathcal{L}_{\theta}$ is the
  generator of a contraction semigroup, we take $(\varphi^m (t))_{t \geqslant
  0}$ to be the solution to the integral equation
  \begin{equation}
    \varphi^{m, h} (t) = e^{\mathcal{L}_{\theta} t} \varphi_0 + \int_0^t
    e^{\mathcal{L}_{\theta} (t - s)} \mathcal{G}^{m, h} \varphi^{m, h} (s)
    \mathd s \label{eq:duhamel-mh}
  \end{equation}
  in $C (\mathbb{R}_+ ; (1 -\mathcal{L}_{\theta}) \mathcal{H})$ and deduce
  easily that $\varphi^{m, h}$ solves the equation $\partial_t \varphi^{m, h}
  (t) = (\mathcal{L}_{\theta} +\mathcal{G}^{m, h}) \varphi^{m, h} (t)$.
  Moreover
  \[ \| (1 -\mathcal{L}_{\theta}) (1 +\mathcal{N})^{2 p} \varphi^{m, h} (t) \|
     \leqslant C_{t, h, m} \| (1 -\mathcal{L}_{\theta}) (1 +\mathcal{N})^{2 p}
     \varphi_0 \|, \]
  for any finite $t \geqslant 0$ and $p > 0$ but not uniformly in $h$ and~$m$.
  Now
  \begin{eqnarray*}
    &  & \langle (1 +\mathcal{N})^{2 p} \varphi^{m, h} (t), \mathcal{G}^{m,
    h} \varphi^{m, h} (t) \rangle\\
    & = & \langle (1 +\mathcal{N})^{2 p} \varphi^{m, h} (t), \mathcal{G}^{m,
    h}_+ \varphi^{m, h} (t) \rangle - \langle \mathcal{G}^{m, h}_+ (1
    +\mathcal{N})^{2 p} \varphi^{m, h} (t), \varphi^{m, h} (t) \rangle\\
    & = & \langle (1 +\mathcal{N})^{2 p} \varphi^{m, h} (t), \mathcal{G}^{m,
    h}_+ \varphi^{m, h} (t) \rangle - \langle \mathcal{N}^{2 p}
    \mathcal{G}^{m, h}_+ \varphi^{m, h} (t), \varphi^{m, h} (t) \rangle\\
    & = & \langle ((1 +\mathcal{N})^{2 p} -\mathcal{N}^{2 p}) \varphi^{m, h}
    (t), \mathcal{G}^{m, h}_+ \varphi^{m, h} (t) \rangle .
  \end{eqnarray*}
  Using $| (1 +\mathcal{N})^{2 p} -\mathcal{N}^{2 p} | \lesssim (1
  +\mathcal{N})^{2 p - 1}$ and the uniform estimates in
  Lemma~\ref{lem:uniform-bds-G} we have that, for some $\sigma \in (0, 1 /
  2)$,
  \begin{eqnarray*}
    &  & | \langle ((1 +\mathcal{N})^{2 p} -\mathcal{N}^{2 p}) \varphi^{m, h}
    (t), \mathcal{G}^{m, h}_+ \varphi^{m, h} (t) \rangle |\\
    & \leqslant & \| (1 +\mathcal{N})^p (1 -\mathcal{L}_{\theta})^{\sigma}
    \varphi^{m, h} (t) \| \| (1 +\mathcal{N})^{p - 1} (1
    -\mathcal{L}_{\theta})^{- \sigma} \mathcal{G}^{m, h}_+ \varphi^{m, h} (t)
    \|\\
    & \leqslant & \| (1 +\mathcal{N})^p (1 -\mathcal{L}_{\theta})^{\sigma}
    \varphi^{m, h} (t) \|^2 .
  \end{eqnarray*}
  Therefore,
  \[ | \langle (1 +\mathcal{N})^{2 p} \varphi^{m, h} (t), \mathcal{G}^{m, h}
     \varphi^{m, h} (t) \rangle | \lesssim \| (1 +\mathcal{N})^p (1
     -\mathcal{L}_{\theta})^{\sigma} \varphi^{m, h} (t) \|^2 \]
  and by interpolation we can bound this by
  \[ | \langle (1 +\mathcal{N})^{2 p} \varphi^{m, h} (t), \mathcal{G}^{m, h}
     \varphi^{m, h} (t) \rangle | \leqslant C_{\delta} \| (1 +\mathcal{N})^p
     \varphi^{m, h} (t) \|^2 + \delta \| (1 +\mathcal{N})^p (1
     -\mathcal{L}_{\theta})^{1 / 2} \varphi^{m, h} (t) \|^2, \]
  for some small $\delta > 0$. Therefore, we have
  \begin{eqnarray*}
    &  & \partial_t \frac{1}{2} \| (1 +\mathcal{N})^p \varphi^{m, h} (t) \|^2
    = \langle (1 +\mathcal{N})^{2 p} \varphi^{m, h} (t), (\mathcal{L}_{\theta}
    +\mathcal{G}^{m, h}) \varphi^{m, h} (t) \rangle\\
    & = & - \| (1 +\mathcal{N})^p (1 -\mathcal{L}_{\theta})^{1 / 2}
    \varphi^{m, h} (t) \|^2 + \| (1 +\mathcal{N})^p \varphi^{m, h} (t) \|^2 +
    \langle (1 +\mathcal{N})^{2 p} \varphi^{m, h} (t), \mathcal{G}^{m, h}
    \varphi^{m, h} (t) \rangle\\
    & \leqslant & - (1 - \delta) \| (1 +\mathcal{N})^p (1
    -\mathcal{L}_{\theta})^{1 / 2} \varphi^{m, h} (t) \|^2 + C_{\delta}' \| (1
    +\mathcal{N})^p \varphi^{m, h} (t) \|^2
  \end{eqnarray*}
  uniformly in $m$ and~$h$. Integrating this inequality gives
  \[ \| (1 +\mathcal{N})^p \varphi^{m, h} (t) \|^2 + \int_0^t e^{- C (t - s)}
     \| (1 +\mathcal{N})^p (1 -\mathcal{L}_{\theta})^{1 / 2} \varphi^{m, h}
     (s) \|^2 \mathd s \lesssim e^{C t} \| (1 +\mathcal{N})^p \varphi_0 \|^2
  \]
  for all $p \geqslant 1$ where the constants are uniform in $m$ and~$h$.
  Inserting this a priori bound in the mild formulation in
  eq.~{\eqref{eq:duhamel-mh}} we obtain
  \[ \| (1 +\mathcal{N})^p (1 -\mathcal{L}_{\theta}) \varphi^{m, h} (t) \|
     \leqslant \| (1 +\mathcal{N})^p (1 -\mathcal{L}_{\theta}) \varphi^{m, h}
     (0) \| + \int_0^t \| (1 +\mathcal{N})^{p + 1} (1
     -\mathcal{L}_{\theta})^{1 / 2} \varphi^{m, h} (s) \| \mathd s \]
  \[ \lesssim_t \| (1 +\mathcal{N})^p (1 -\mathcal{L}_{\theta}) \varphi_0 \| +
     \| (1 +\mathcal{N})^{p + 1} \varphi_0 \|, \]
  where we also used that
  \begin{eqnarray*}
    \| (1 +\mathcal{N})^p (1 -\mathcal{L}_{\theta}) \mathcal{G}^{m, h}
    \varphi^{m, h} (s) \| & \leqslant & C (m) \| (1 +\mathcal{N})^{p + 1}
    \mathcal{G}^{m, h} \varphi^{m, h} (s) \|\\
    & \lesssim_m & \| (1 +\mathcal{N})^{p + 1} (1 -\mathcal{L}_{\theta})^{1 /
    2} \varphi^{m, h} (s) \|
  \end{eqnarray*}
  by the presence of the Galerkin projectors and our (non-uniform) bounds.
  Indeed, note that
  \[ (1 -\mathcal{L}_{\theta}) \Pi_m \lesssim | m |^{2 \theta} (1
     +\mathcal{N}) \Pi_m . \]
  We conclude that
  \[ \| (1 +\mathcal{N})^p (1 -\mathcal{L}_{\theta}) \varphi^{m, h} (t) \|
     \lesssim_{t, m} \| (1 +\mathcal{N})^{p + 1} (1 -\mathcal{L}_{\theta})
     \varphi_0 \|, \]
  uniformly in~$h$. We can then pass to the limit (by subsequence) as $h
  \rightarrow 0$ and obtain a function $\varphi^m \in C (\mathbb{R}_+, (1
  +\mathcal{N})^{- p} (1 -\mathcal{L}_{\theta})^{- 1} \mathcal{H})$ satisfying
  the estimates
  \[ \| (1 +\mathcal{N})^p \varphi^m (t) \|^2 + \int_0^t e^{- C (t - s)} \| (1
     +\mathcal{N})^p (1 -\mathcal{L}_{\theta})^{1 / 2} \varphi^m (s) \|^2
     \mathd s \lesssim e^{C t} \| (1 +\mathcal{N})^p \varphi_0 \|^2 \]
  and
  \[ \| (1 +\mathcal{N})^p (1 -\mathcal{L}_{\theta}) \varphi^m (t) \|
     \lesssim_{t, m} \| (1 +\mathcal{N})^{p + 1} (1 -\mathcal{L}_{\theta})
     \varphi_0 \|, \]
  for all $t \geqslant 0$ and $p \geqslant 1$. As a consequence, $\varphi^m
  \in C (\mathbb{R}_+, \mathcal{D} (\mathcal{L}^m))$ for all $t \geqslant 0$
  as soon as $\| (1 +\mathcal{N})^2 (1 -\mathcal{L}_{\theta}) \varphi_0 \| <
  \infty$. By passing to the limit in the equation, $\varphi^m$ also satisfies
  \[ \partial_t \varphi^m (t) = (\mathcal{L}_{\theta} +\mathcal{G}^m)
     \varphi^m (t) =\mathcal{L}^m \varphi^m (t) . \]
\end{proof}

Recall that we write $T^m$ to indicate the semigroup generated by the Galerkin
approximation~$\omega^m$. Moreover, if we denote by $\hat{\mathcal{L}}^m$ its
Hille--Yosida generator, we have the following result.

\begin{lemma}
  \label{eq:lemma-lm}$(\mathcal{L}^m, \mathcal{D} (\mathcal{L}^m))$ is
  closable and its closure is the generator $\hat{\mathcal{L}}^m$. In
  particular, if $\varphi \in \mathcal{V}$, then $\varphi^m (t) = T^m_t
  \varphi$ solves
  \[ \partial_t \varphi^m (t) =\mathcal{L}^m \varphi^m (t), \]
  and we have
  \[ \mathcal{L}^m T^m_t \varphi = T^m_t \mathcal{L}^m \varphi . \]
\end{lemma}

\begin{proof}
  Let $(\omega^m_t)_{t \geqslant 0}$ be a solution to the Galerkin
  approximation~{\eqref{eq:ns-vort-appr}} with initial condition~$\omega_0$.
  If $\varphi \in \mathcal{C}$ is a cylinder function, then we have
  \[ T^m_t \varphi (\omega_0) - \varphi (\omega_0) =\mathbb{E}_{\omega_0}
     \left[ \int_0^t \mathcal{L}^m \varphi (\omega^m_s) \mathd s \right] =
     \int_0^t T^m_s (\mathcal{L}^m \varphi) (\omega_0) \mathd s. \]
  By approximation (using a Bochner integral in $\mathcal{H}$ on the
  right-hand side), we can extend this point-wise formula to all $\varphi \in
  \mathcal{D} (\mathcal{L}^m)$ obtaining for them that $T^m_t \varphi -
  \varphi = \int_0^t T^m_s \mathcal{L}^m \varphi \mathd s$ in~$\mathcal{H}$.
  For every $\varphi \in \mathcal{D} (\mathcal{L}^m)$,
  Lemma~\ref{lem:semigroup} implies that the map $s \mapsto T^m_s
  \mathcal{L}^m \varphi \in \mathcal{H}$ is continuous, and therefore
  \[ \frac{T_t^m \varphi - \varphi}{t} \rightarrow \mathcal{L}^m \varphi,
     \qquad \text{as } t \rightarrow 0, \qquad \varphi \in \mathcal{D}
     (\mathcal{L}^m), \]
  with convergence in~$\mathcal{H}$. As a consequence, $\varphi \in
  \mathcal{D} (\hat{\mathcal{L}}^m)$ and we conclude that
  $\hat{\mathcal{L}}^m$ is an extension of $(\mathcal{L}^m, \mathcal{D}
  (\mathcal{L}^m))$. By Lemma~\ref{lemma:closure}, we have that the closure of
  $\mathcal{L}^m$ is $\hat{\mathcal{L}}^m$ and that $\mathcal{L}^m T^m_t
  \varphi = T^m_t \mathcal{L}^m \varphi$ for all $\varphi \in \mathcal{V}$.
\end{proof}

Using the commutation $\mathcal{L}^m T^m_t \varphi = T^m_t \mathcal{L}^m
\varphi$, we are able to get better estimates, uniform in~$m$.

\begin{corollary}
  \label{cor:est-phi}For all $\varphi_0 \in \mathcal{V}$ and for all $\alpha
  \geqslant 1$, we have
  \begin{equation}
    \| (1 +\mathcal{N})^{\alpha} \partial_t \varphi^m (t) \|^2 = \| (1
    +\mathcal{N})^{\alpha} \mathcal{L}^m \varphi^m (t) \|^2 \lesssim e^{t C}
    \| (1 +\mathcal{N})^{\alpha} \mathcal{L}^m \varphi_0 \|^2,
    \label{eq:ctrl-dtphi}
  \end{equation}
  and
  \begin{equation}
    \| (1 +\mathcal{N})^{\alpha} (1 -\mathcal{L}_{\theta})^{1 / 2} \varphi^m
    (t) \|^2 \lesssim t e^{t C} \| (1 +\mathcal{N})^{\alpha} \mathcal{L}^m
    \varphi^m_0 \|^2 + \| (1 +\mathcal{N})^{\alpha} (1
    -\mathcal{L}_{\theta})^{1 / 2} \varphi_0 \|^2 . \label{eq:ctrl-phi}
  \end{equation}
\end{corollary}

\begin{proof}
  Recall $T^m_t \varphi^m_0 = \varphi^m (t)$. We already know
  \[ e^{- t C} \| (1 +\mathcal{N})^{\alpha} T^m_t \varphi^m_0 \|^2 +
     \int_0^{\infty} e^{- s C} \| (1 +\mathcal{N})^{\alpha} (1
     -\mathcal{L}_{\theta})^{1 / 2} T^m_s \varphi^m_0 \|^2 \mathd s \lesssim
     \| (1 +\mathcal{N})^{\alpha} \varphi_0 \|^2, \]
  which yields
  \begin{eqnarray*}
    \int_0^{\infty} e^{- t C} \| (1 +\mathcal{N})^{\alpha} (1
    -\mathcal{L}_{\theta})^{1 / 2} \partial_t T^m_t \varphi^m_0 \|^2 \mathd t
    & = & \int_0^{\infty} e^{- t C} \| (1 +\mathcal{N})^{\alpha} (1
    -\mathcal{L}_{\theta})^{1 / 2} T^m_t \mathcal{L}^m \varphi_0 \|^2 \mathd
    t\\
    & \lesssim & \| (1 +\mathcal{N})^{\alpha} \mathcal{L}^m \varphi_0 \|^2,
  \end{eqnarray*}
  and
  \[ \begin{array}{ll}
       & \| (1 +\mathcal{N})^{\alpha} (1 -\mathcal{L}_{\theta})^{1 / 2} T^m_t
       \varphi_0 \|^2\\
       \lesssim & \left\| \int_0^t (1 +\mathcal{N})^{\alpha} (1
       -\mathcal{L}_{\theta})^{1 / 2} \partial_s T^m_s \varphi^m_0 \mathd s
       \right\|^2 + \| (1 +\mathcal{N})^{\alpha} (1 -\mathcal{L}_{\theta})^{1
       / 2} \varphi_0 \|^2\\
       \leqslant & t \int_0^t \| (1 +\mathcal{N})^{\alpha} (1
       -\mathcal{L}_{\theta})^{1 / 2} \partial_s T^m_s \varphi_0 \|^2 \mathd s
       + \| (1 +\mathcal{N})^{\alpha} (1 -\mathcal{L}_{\theta})^{1 / 2}
       \varphi_0 \|^2\\
       \leqslant & t e^{t C} \int_0^t e^{- s C} \| (1 +\mathcal{N})^{\alpha}
       (1 -\mathcal{L}_{\theta})^{1 / 2} \partial_s T^m_s \varphi_0 \|^2
       \mathd s + \| (1 +\mathcal{N})^{\alpha} (1 -\mathcal{L}_{\theta})^{1 /
       2} \varphi_0 \|^2\\
       \lesssim & t e^{t C} \| (1 +\mathcal{N})^{\alpha} \mathcal{L}^m
       \varphi_0 \|^2 + \| (1 +\mathcal{N})^{\alpha} (1
       -\mathcal{L}_{\theta})^{1 / 2} \varphi_0 \|^2,
     \end{array} \]
  which is what claimed.
\end{proof}

\subsection{Controlled structures}

The a priori bounds~{\eqref{eq:ctrl-dtphi}} and~{\eqref{eq:ctrl-phi}} bring us
in position to control $\| \varphi^m (t) \|$, $\| \partial_t \varphi^m (t)
\|$, and $\| \mathcal{L}^m \varphi^m (t) \|$ uniformly in $m$ and locally
uniformly in~$t$, but in order to study the limiting Kolmogorov backward
equation we have first to deal with the limiting operator~$\mathcal{L}$ and to
define a domain~$\mathcal{D} (\mathcal{L})$.

To take care of the term $\mathcal{G}$ in the limiting operator $\mathcal{L}$,
we decompose it by means of a cut-off function $\mathcal{M}= M (\mathcal{N})$
as follows
\[ \mathcal{G}^m = \mathbbm{1}_{| \mathcal{L}_{\theta} | \geqslant
   \mathcal{M}} \mathcal{G}^m + \mathbbm{1}_{| \mathcal{L}_{\theta} |
   <\mathcal{M}} \mathcal{G}^m \backassign \mathcal{G}^{m, \succ}
   +\mathcal{G}^{m, \prec} . \]
We then set
\begin{equation}
  \varphi^{m, \sharp} \assign \varphi^m - (1 -\mathcal{L}_{\theta})^{- 1}
  \mathcal{G}^{m, \succ} \varphi^m, \label{eq:phimsharp}
\end{equation}
so that
\begin{equation}
  (1 -\mathcal{L}^m) \varphi^m = (1 -\mathcal{L}_{\theta}) \varphi^{m, \sharp}
  +\mathcal{G}^{m, \prec} \varphi^m . \label{eq:act-Lm}
\end{equation}
\begin{lemma}
  \label{lem:exist-controlled}Let $w$ be a weight, $L \geqslant 1$,
  $\bar{\varepsilon} \in] 0, (\theta - 1) / (2 \theta) [$ and $M (n) = L (n +
  1)^{3 \theta / (\theta - 1 - 2 \theta \varepsilon)}$. Then we have
  \begin{equation}
    \| w (\mathcal{N}) (1 -\mathcal{L}_{\theta})^{- 1 / 2} \mathcal{G}^{m,
    \succ} \psi \| \lesssim | w | L^{- \frac{(\theta - 1)}{2 \theta} +
    \bar{\varepsilon}} \| w (\mathcal{N}) (1 -\mathcal{L}_{\theta})^{1 / 2}
    \psi \| . \label{eq:est-G-maj}
  \end{equation}
  Consequently, there exists $L_0 = L_0 (| w |)$ such that, for all $L
  \geqslant L_0$ and all $\varphi^{\sharp} \in w (\mathcal{N})^{- 1} (1
  -\mathcal{L}_{\theta})^{- 1 / 2} \mathcal{H}$, there is a unique $\varphi^m
  =\mathcal{K}^m \varphi^{\sharp}$ such that
  \[ \varphi^m = (1 -\mathcal{L}_{\theta})^{- 1} \mathcal{G}^{m, \succ}
     \varphi^m + \varphi^{\sharp} \in w (\mathcal{N})^{- 1} (1
     -\mathcal{L}_{\theta})^{- 1 / 2} \mathcal{H}, \]
  which satisfies the bound
  \begin{equation}
    \begin{array}{ll}
      & \| w (\mathcal{N}) (1 -\mathcal{L}_{\theta})^{1 / 2} \mathcal{K}^m
      \varphi^{\sharp} \| + | w |^{- 1} L^{(\theta - 1) / (2 \theta) -
      \bar{\varepsilon}} \| w (\mathcal{N}) (1 -\mathcal{L}_{\theta})^{1 / 2}
      (\mathcal{K}^m \varphi^{\sharp} - \varphi^{\sharp}) \|\\
      \lesssim & \| w (\mathcal{N}) (1 -\mathcal{L}_{\theta})^{1 / 2}
      \varphi^{\sharp} \| .
    \end{array} \label{eq:controlled-est}
  \end{equation}
  All the estimates are uniform in $m$ and true in the limit $m \rightarrow
  \infty$. We denote $\mathcal{K}=\mathcal{K}^{\infty}$.
\end{lemma}

\begin{proof}
  We start with the estimate on~$\mathcal{G}_+^{m, \succ}$. We have, for
  $\varepsilon \in] 0, 1 / 2 - 1 / (2 \theta) [$, using
  Lemma~\ref{lem:uniform-bds-G},
  \begin{eqnarray*}
    \| w (\mathcal{N}) (1 -\mathcal{L}_{\theta})^{- 1 / 2} \mathcal{G}_+^{m,
    \succ} \psi \| & \lesssim & \| w (\mathcal{N}) (1
    -\mathcal{L}_{\theta})^{- 1 / 2} \mathbbm{1}_{| \mathcal{L}_{\theta} |
    \geqslant M (\mathcal{N})}^{} \mathcal{G}_+^m \psi \|\\
    & \lesssim & \| w (\mathcal{N}) M (\mathcal{N})^{- 1 / 2 + 1 / (2 \theta)
    + \varepsilon} (1 -\mathcal{L}_{\theta})^{- 1 / (2 \theta) - \varepsilon}
    \mathcal{G}_+^m \psi \|\\
    & \lesssim & \| w (\mathcal{N}+ 1) M (\mathcal{N}+ 1)^{- 1 / 2 + 1 / (2
    \theta) + \varepsilon} (\mathcal{N}+ 1) (1 -\mathcal{L}_{\theta})^{1 / 2 -
    \varepsilon} \psi \| .
  \end{eqnarray*}
  The bound on $\mathcal{G}_-^{m, \succ}$ can be obtained using again
  Lemma~\ref{lem:uniform-bds-G}:
  \begin{eqnarray*}
    \| w (\mathcal{N}) (1 -\mathcal{L}_{\theta})^{- 1 / 2} \mathcal{G}_-^{m,
    \succ} \psi \| & \lesssim & \| w (\mathcal{N}) (1
    -\mathcal{L}_{\theta})^{- 1 / 2} \mathbbm{1}_{| \mathcal{L}_{\theta} |
    \geqslant M (\mathcal{N})}^{} \mathcal{G}_-^m \psi \|\\
    & \lesssim & \| w (\mathcal{N}) M (\mathcal{N})^{- 1 / 2 + 1 / (2
    \theta)} (1 -\mathcal{L}_{\theta})^{- 1 / (2 \theta)} \mathcal{G}_-^m \psi
    \|\\
    & \lesssim & \| w (\mathcal{N}- 1) M (\mathcal{N}- 1)^{- 1 / 2 + 1 / (2
    \theta)} \mathcal{N}^{3 / 2} (1 -\mathcal{L}_{\theta})^{1 / 2} \psi \| .
  \end{eqnarray*}
  In conclusion, for $\varepsilon \in] 0, (\theta - 1) / (2 \theta) [$,
  choosing $M (n) = L (n + 1)^{3 \theta / (\theta - 1 - 2 \theta
  \varepsilon)}$, for $L \geqslant 1$,
  \[ \| w (\mathcal{N}) (1 -\mathcal{L}_{\theta})^{- 1 / 2} \mathcal{G}^{m,
     \succ} \psi \| \lesssim L^{- 1 / 2 + 1 / (2 \theta) + \varepsilon} | w |
     \| w (\mathcal{N}) (1 -\mathcal{L}_{\theta})^{1 / 2} \psi \| . \]
  Now let $\varphi^{\sharp} \in w (\mathcal{N})^{- 1} (1
  -\mathcal{L}_{\theta})^{- 1 / 2} \mathcal{H}$, the map
  \[ \begin{array}{c}
       \Psi^m : w (\mathcal{N})^{- 1} (1 -\mathcal{L}_{\theta})^{- 1 / 2}
       \mathcal{H} \rightarrow w (\mathcal{N})^{- 1} (1
       -\mathcal{L}_{\theta})^{- 1 / 2} \mathcal{H},\\
       \psi \mapsto \Psi^m (\psi) \assign (1 -\mathcal{L}_{\theta})^{- 1}
       \mathcal{G}^{m, \succ} \psi + \varphi^{\sharp},
     \end{array} \]
  satisfies, for some positive constant~$C$,
  \begin{eqnarray*}
    \| w (\mathcal{N}) (1 -\mathcal{L}_{\theta})^{1 / 2} \Psi^m (\psi) \| &
    \leqslant & \| w (\mathcal{N}) (1 -\mathcal{L}_{\theta})^{- 1 / 2}
    \mathcal{G}^{m, \succ} \psi \| + \| w (\mathcal{N}) (1
    -\mathcal{L}_{\theta})^{1 / 2} \varphi^{\sharp} \|\\
    & \leqslant & C L^{- 1 / 2 + 1 / (2 \theta) + \varepsilon} | w | \| w
    (\mathcal{N}) (1 -\mathcal{L}_{\theta})^{1 / 2} \psi \| + \| w
    (\mathcal{N}) (1 -\mathcal{L}_{\theta})^{1 / 2} \varphi^{\sharp} \| .
  \end{eqnarray*}
  Namely, $\Psi^m$ is well-defined and, choosing $L$ large enough, it is a
  contraction leaving the ball of radius $2 \| w (\mathcal{N}) (1
  -\mathcal{L}_{\theta})^{1 / 2} \varphi^{\sharp} \|$ invariant. Therefore, it
  has a unique fixed point $\mathcal{K}^m \varphi^{\sharp}$ satisfying the
  claimed inequalities.
\end{proof}

\begin{remark}
  In the previous lemma, the cut-off $M (n)$ depends via $| w |$ on the
  weight~$w$. In the following we will only use \tmtextit{polynomial weights}
  of the form $w (n) = (1 + n)^{\alpha}$ with $| \alpha | \leqslant K$ for a
  fixed $K$. In this case $| w |$ is uniformly bounded and it is possible to
  select a cut-off which is adapted to all those weights. This will be fixed
  once and for all and not discussed further.
\end{remark}

\begin{proposition}
  \label{prop:well-def-gen}Let $w$ be a polynomial weight, $\gamma \geqslant
  0$, $\bar{\varepsilon}$ as in Lemma~\ref{lem:exist-controlled},
  \[ \alpha (\gamma) = \frac{\theta (6 \gamma + 5) - 2}{2 (\theta - 1)} . \]
  Let
  \[ \varphi^{\sharp} \in w (\mathcal{N})^{- 1} (1 -\mathcal{L}_{\theta})^{-
     1} \mathcal{H} \cap w (\mathcal{N})^{- 1} (1 +\mathcal{N})^{- \alpha
     (\gamma)} (1 -\mathcal{L}_{\theta})^{- 1 / 2} \mathcal{H}, \]
  and set $\varphi^m \assign \mathcal{K}^m \varphi^{\sharp}$. Then
  $\mathcal{L}^m \varphi^m$ is a well-defined operator and we have the bound
  \begin{equation}
    \| w (\mathcal{N}) (1 -\mathcal{L}_{\theta})^{\gamma} \mathcal{G}^{m,
    \prec} \varphi^m \| \lesssim \| w (\mathcal{N}) (1 +\mathcal{N})^{\alpha
    (\gamma)} (1 -\mathcal{L}_{\theta})^{1 / 2} \varphi^{\sharp} \| .
    \label{eq:est-G-minor}
  \end{equation}
\end{proposition}

\begin{proof}
  By eq.~{\eqref{eq:act-Lm}} we need only to estimate $\mathcal{G}^{m, \prec}
  \varphi^m$. We first deal with $\mathcal{G}_+^{m, \prec}$: we have
  by~{\eqref{eq:unifbdG1}}, for $\delta < 1 / 2 - 1 / (2 \theta)$,
  \begin{eqnarray*}
    \| w (\mathcal{N}) (1 -\mathcal{L}_{\theta})^{\gamma} \mathcal{G}_+^{m,
    \prec} \varphi^m \| & = & \| w (\mathcal{N}) (1
    -\mathcal{L}_{\theta})^{\gamma} \mathbbm{1}_{| \mathcal{L}_{\theta} | < M
    (\mathcal{N})} \mathcal{G}_+^m \varphi^m \|\\
    & \lesssim & \| w (\mathcal{N}) M (\mathcal{N})^{\gamma + 1 / 2 - \delta}
    (1 -\mathcal{L}_{\theta})^{- 1 / 2 + \delta} \mathcal{G}_+^m \varphi^m
    \|\\
    & \lesssim & \| w (\mathcal{N}+ 1) M (\mathcal{N}+ 1)^{\gamma + 1 / 2}
    (\mathcal{N}+ 1) (1 -\mathcal{L}_{\theta})^{1 / (2 \theta) + \delta}
    \varphi^m \| .
  \end{eqnarray*}
  For $\mathcal{G}_-^{m, \prec}$, it follows in a similar way from
  estimate~{\eqref{eq:unifbdG2}} that, for every $\delta \in] 0, 1 / (2
  \theta)]$,
  \begin{eqnarray*}
    \| w (\mathcal{N}) (1 -\mathcal{L}_{\theta})^{\gamma} \mathcal{G}_-^{m,
    \prec} \varphi^m \| & = & \| w (\mathcal{N}) (1
    -\mathcal{L}_{\theta})^{\gamma} \mathbbm{1}_{| \mathcal{L}_{\theta} | < M
    (\mathcal{N})} \mathcal{G}_-^m \varphi^m \|\\
    & \lesssim & \| w (\mathcal{N}) M (\mathcal{N})^{\gamma + 1 / (2 \theta)}
    (1 -\mathcal{L}_{\theta})^{- 1 / (2 \theta)} \mathcal{G}_-^m \varphi^m
    \|\\
    & \lesssim & \| w (\mathcal{N}- 1) M (\mathcal{N}- 1)^{\gamma + 1 / (2
    \theta)} \mathcal{N}^{3 / 2} (1 -\mathcal{L}_{\theta})^{1 / 2} \varphi^m
    \| .
  \end{eqnarray*}
  These bounds and the definition of $M (n)$ give the claimed bound
  on~$\mathcal{G}^{m, \prec}$.
\end{proof}

\subsection{Limiting generator and its domain}

\begin{lemma}
  \label{lem:density}Let $w$ be a weight and take a cut-off function as in
  Proposition~\ref{prop:well-def-gen} with $\gamma = 0$. Set
  \[ \mathcal{D}_w (\mathcal{L}) \assign \{ \mathcal{K} \varphi^{\sharp} :
     \varphi^{\sharp} \in w (\mathcal{N})^{- 1} (1 -\mathcal{L}_{\theta})^{-
     1} \mathcal{H} \cap w (\mathcal{N})^{- 1} (\mathcal{N}+ 1)^{- \alpha (0)}
     (1 -\mathcal{L}_{\theta})^{- 1 / 2} \mathcal{H} \} . \]
  Then $\mathcal{D}_w (\mathcal{L})$ is dense in $w (\mathcal{N})^{- 1}
  \mathcal{H}$. If $w \equiv 1$ we simply write $\mathcal{D} (\mathcal{L})$.
\end{lemma}

\begin{proof}
  Note that $w (\mathcal{N})^{- 1} (1 -\mathcal{L}_{\theta})^{- 1} \mathcal{H}
  \cap w (\mathcal{N})^{- 1} (\mathcal{N}+ 1)^{- \alpha (0)} (1
  -\mathcal{L}_{\theta})^{- 1 / 2} \mathcal{H}$ is dense in $w
  (\mathcal{N})^{- 1} \mathcal{H}$, therefore, in order to prove
  Lemma~\ref{lem:density}, it suffices to show that, for any $\psi \in w
  (\mathcal{N})^{- 1} (1 -\mathcal{L}_{\theta})^{- 1} \mathcal{H} \cap w
  (\mathcal{N})^{- 1} (\mathcal{N}+ 1)^{- \alpha (0)} (1
  -\mathcal{L}_{\theta})^{- 1 / 2} \mathcal{H}$ and for all $\nu \geqslant 1$,
  there exists $\varphi^{\nu} \in \mathcal{D}_w (\mathcal{L})$ such that
  \begin{eqnarray}
    \| w (\mathcal{N}) (1 -\mathcal{L}_{\theta})^{1 / 2} (\varphi^{\nu} -
    \psi) \| & \lesssim & \nu^{- (\theta - 1) / (2 \theta) +
    \bar{\varepsilon}} \| w (\mathcal{N}) (1 -\mathcal{L}_{\theta})^{1 / 2}
    \psi \|,  \label{eq:density-e1}\\
    \| w (\mathcal{N}) (1 -\mathcal{L}_{\theta})^{1 / 2} \varphi^{\nu} \| &
    \lesssim & \| w (\mathcal{N}) (1 -\mathcal{L}_{\theta})^{1 / 2} \psi \|, 
    \label{eq:density-e2}\\
    \| w (\mathcal{N}) (1 -\mathcal{L}) \varphi^{\nu} \| & \lesssim & \nu^{1 /
    (2 \theta) + \delta} (\| w (\mathcal{N}) (1 -\mathcal{L}_{\theta}) \psi \|
    \nobracket  \label{eq:density-e3}\\
    &  & \nobracket + \| w (\mathcal{N}) (\mathcal{N}+ 1)^{\alpha (0)} (1
    -\mathcal{L}_{\theta})^{1 / 2} \psi \|), \nonumber
  \end{eqnarray}
  for some $\delta > 0$. By Lemma~\ref{lem:exist-controlled}, there exists
  $\varphi^{\nu} \in w (\mathcal{N})^{- 1} \mathcal{H}$ such that
  \[ \varphi^{\nu}_{} = \mathbbm{1}_{\nu M (\mathcal{N}) \leqslant |
     \mathcal{L}_{\theta} |} (1 -\mathcal{L}_{\theta})^{- 1} \mathcal{G}
     \varphi^{\nu} + \psi \]
  and satisfying estimates~{\eqref{eq:density-e1}}--{\eqref{eq:density-e2}}.
  We are left to show that $\varphi^{\nu} \in \mathcal{D}_w (\mathcal{L})$
  and~{\eqref{eq:density-e3}}. Note that
  \[ \varphi^{\nu}_{} = (1 -\mathcal{L}_{\theta})^{- 1} \mathcal{G}^{\succ}
     \varphi^{\nu} + \varphi^{\nu, \sharp}, \]
  where
  \[ \varphi^{\nu, \sharp} = \psi - \mathbbm{1}_{M (\mathcal{N}) \leqslant |
     \mathcal{L}_{\theta} | < \nu M (\mathcal{N})} (1
     -\mathcal{L}_{\theta})^{- 1} \mathcal{G} \varphi^{\nu} . \]
  In particular, we have $\mathcal{L} \varphi^{\nu} = \varphi^{\nu}
  +\mathcal{G}^{\prec} \varphi^{\nu} - (1 -\mathcal{L}_{\theta}) \varphi^{\nu,
  \sharp}$, and, by Proposition~\ref{prop:well-def-gen}, it suffices to
  estimate $\varphi^{\nu, \sharp}$ in $w (\mathcal{N})^{- 1} (1
  -\mathcal{L}_{\theta})^{- 1} \mathcal{H} \cap w (\mathcal{N})^{- 1}
  (\mathcal{N}+ 1)^{- \alpha (0)} (1 -\mathcal{L}_{\theta})^{- 1 / 2}
  \mathcal{H}$. The first contribution, $\psi$, satisfies the required bounds
  by assumption, so it is enough to show that the second contribution, which
  we denote by $\psi^{\nu}$, satisfies
  \begin{eqnarray}
    \| w (\mathcal{N}) (1 -\mathcal{L}_{\theta}) \psi^{\nu} \| & \lesssim &
    \nu^{1 / (2 \theta) + \delta} \| w (\mathcal{N}) (\mathcal{N}+ 1)^{\alpha
    (0)} (1 -\mathcal{L}_{\theta})^{1 / 2} \psi \|,  \label{eq:density-e4}\\
    \| w (\mathcal{N}) (\mathcal{N}+ 1)^{\alpha (0)} (1
    -\mathcal{L}_{\theta})^{1 / 2} \psi^{\nu} \| & \lesssim & \| w
    (\mathcal{N}) (\mathcal{N}+ 1)^{\alpha (0)} (1 -\mathcal{L}_{\theta})^{1 /
    2} \psi \| .  \label{eq:density-e5}
  \end{eqnarray}
  Notice that $(1 -\mathcal{L}_{\theta}) \psi^{\nu} = - \mathbbm{1}_{M
  (\mathcal{N}) \leqslant | \mathcal{L}_{\theta} | < \nu M (\mathcal{N})}
  \mathcal{G} \varphi^{\nu}$, hence estimate~{\eqref{eq:density-e4}} can be
  obtained from the uniform bounds in Lemma~\ref{lem:uniform-bds-G} as follows
  (note that those bounds are valid also when $m = + \infty$). We have,
  for~$\mathcal{G}_+$,
  \[ \begin{array}{ll}
       & \| w (\mathcal{N})  \mathbbm{1}_{M (\mathcal{N}) \leqslant |
       \mathcal{L}_{\theta} | < \nu M (\mathcal{N})} \mathcal{G}_+ \varphi
       \|\\
       \lesssim & \| w (\mathcal{N}) (1 -\mathcal{L}_{\theta})^{1 / (2 \theta)
       + \delta}  \mathbbm{1}_{M (\mathcal{N}) \leqslant |
       \mathcal{L}_{\theta} | < \nu M (\mathcal{N})}  (1
       -\mathcal{L}_{\theta})^{- 1 / (2 \theta) - \delta} \mathcal{G}_+
       \varphi \|\\
       \lesssim & \nu^{1 / (2 \theta) + \delta} \| w (\mathcal{N}+ 1) M
       (\mathcal{N}+ 1)^{1 / (2 \theta) + \delta} (\mathcal{N}+ 1) (1
       -\mathcal{L}_{\theta})^{1 / 2 - \delta} \varphi \| .
     \end{array} \]
  For $\mathcal{G}_-$ we have, instead
  \[ \begin{array}{ll}
       & \| w (\mathcal{N}) \mathbbm{1}_{M (\mathcal{N}) \leqslant |
       \mathcal{L}_{\theta} | < \nu M (\mathcal{N})} \mathcal{G}_- \varphi
       \|\\
       \lesssim & \| w (\mathcal{N}) (1 -\mathcal{L}_{\theta})^{1 / (2
       \theta)}  \mathbbm{1}_{M (\mathcal{N}) \leqslant | \mathcal{L}_{\theta}
       | < \nu M (\mathcal{N})}  (1 -\mathcal{L}_{\theta})^{- 1 / (2 \theta)}
       \mathcal{G}_- \varphi \|\\
       \lesssim & \nu^{1 / (2 \theta)} \| w (\mathcal{N}- 1) M (\mathcal{N}-
       1)^{1 / (2 \theta)} \mathcal{N}^{3 / 2}  \mathbbm{1}_{M (\mathcal{N})
       \leqslant | \mathcal{L}_{\theta} | < \nu M (\mathcal{N})}  (1
       -\mathcal{L}_{\theta})^{- 1 / (2 \theta)} \mathcal{G}_- \varphi \|,
     \end{array} \]
  which gives estimate~{\eqref{eq:density-e4}} if we choose
  $\bar{\varepsilon}$ small enough. In order to obtain
  estimate~{\eqref{eq:density-e5}}, note that, for $\kappa \in] 0, (\theta -
  1) / (2 \theta) [$,
  \[ \begin{array}{lll}
       \| w (\mathcal{N}) (\mathcal{N}+ 1)^{\alpha (0)} (1
       -\mathcal{L}_{\theta})^{1 / 2} \psi^{\nu} \| & = & \| w (\mathcal{N})
       (\mathcal{N}+ 1)^{\alpha (0)} (1 -\mathcal{L}_{\theta})^{- 1 / 2} 
       \mathbbm{1}_{M (\mathcal{N}) \leqslant | \mathcal{L}_{\theta} | < \nu M
       (\mathcal{N})} \mathcal{G} \varphi^{\nu} \|\\
       & \lesssim & M (n)^{- (\theta - 1) / \theta + 2 \kappa} \| w
       (\mathcal{N}) (\mathcal{N}+ 1)^{\alpha (0)} (1
       -\mathcal{L}_{\theta})^{- 1 / (2 \theta) - \kappa} \mathcal{G}_+
       \varphi^{\nu} \|\\
       &  & + M (n)^{- (\theta - 1) / \theta + 2 \kappa} \| w (\mathcal{N})
       (\mathcal{N}+ 1)^{\alpha (0)} (1 -\mathcal{L}_{\theta})^{- 1 / (2
       \theta)} \mathcal{G}_- \varphi^{\nu} \|
     \end{array} \]
  Now recall that $M (n) \simeq (n + 1)^{3 \theta / (\theta - 1 - 2 \theta
  \bar{\varepsilon})}$ and get by {\eqref{eq:unifbdG1}}--{\eqref{eq:unifbdG2}}
  the inequality
  \[ \| w (\mathcal{N}) (\mathcal{N}+ 1)^{\alpha (0)} (1
     -\mathcal{L}_{\theta})^{1 / 2} \psi^{\nu} \| \lesssim \| w (\mathcal{N})
     (1 +\mathcal{N})^{\alpha (0)} (1 -\mathcal{L}_{\theta})^{1 / 2}
     \varphi^{\nu} \| . \]
  Applying {\eqref{eq:controlled-est}} yields the result.
\end{proof}

\begin{lemma}
  \label{lem:diss}For any $\varphi \in \mathcal{D} (\mathcal{L})$, we have
  \[ \langle \varphi, \mathcal{L} \varphi \rangle \leqslant 0. \]
  In particular, the operator $(\mathcal{L}, \mathcal{D} (\mathcal{L}))$ is
  dissipative.
\end{lemma}

\begin{proof}
  Notice that $\varphi \in \mathcal{D} (\mathcal{L})$ implies
  $\mathcal{L}_{\theta} \varphi, \mathcal{G} \varphi \in (1
  -\mathcal{L}_{\theta})^{1 / 2} \mathcal{H}$ and $\varphi \in (1
  -\mathcal{L}_{\theta})^{- 1 / 2} (1 +\mathcal{N})^{- 1} \mathcal{H}$. These
  regularities are enough to proceed by approximation and establish that
  \[ \langle \varphi, \mathcal{L} \varphi \rangle = - \langle \varphi,
     (-\mathcal{L}_{\theta}) \varphi \rangle + \langle \varphi, \mathcal{G}
     \varphi \rangle = - \langle \varphi, (-\mathcal{L}_{\theta}) \varphi
     \rangle = - \| (-\mathcal{L}_{\theta})^{1 / 2} \varphi \|^2 \leqslant 0,
  \]
  where we used the anti-symmetry of the form associated to $\mathcal{G}$,
  i.e. $\langle \varphi, \mathcal{G} \varphi \rangle = 0$.
\end{proof}

\subsection{Existence and uniqueness for the Kolmogorov equation}

Having defined a domain for $\mathcal{L}$ it remains to study the Kolmogorov
equation $\partial_t \varphi =\mathcal{L} \varphi$. In particular, we consider
the equation for $\varphi^{m, \sharp}$, which was defined
in~{\eqref{eq:phimsharp}},
\[ \begin{array}{lll}
     \partial_t \varphi^{m, \sharp} + (1 -\mathcal{L}_{\theta}) \varphi^{m,
     \sharp} & = & \mathcal{L}^m \varphi^m + (1 -\mathcal{L}_{\theta})
     \varphi^{m, \sharp} - (1 -\mathcal{L}_{\theta})^{- 1} \mathcal{G}^{m,
     \succ} \partial_t \varphi^m\\
     & = & \varphi^m +\mathcal{G}^{m, \prec} \varphi^m - (1
     -\mathcal{L}_{\theta})^{- 1} \mathcal{G}^{m, \succ} \partial_t
     \varphi^m\\
     & = & \varphi^m +\mathcal{G}^{m, \prec} \varphi^m - (1
     -\mathcal{L}_{\theta})^{- 1} \mathcal{G}^{m, \succ} (\varphi^m
     +\mathcal{G}^{m, \prec} \varphi^m - (1 -\mathcal{L}_{\theta}) \varphi^{m,
     \sharp})\\
     & \backassign & \Phi^{m, \sharp} .
   \end{array} \]
We want to get a suitable bound in terms of $\varphi^{m, \sharp}_0$ for each
term of $\Phi^{m, \sharp}$. The Schauder estimate in
Lemma~\ref{lemma:schauder} will be crucial. We will also need the following
result.

\begin{lemma}
  We have
  \begin{equation}
    \begin{array}{ll}
      & \| (1 +\mathcal{N})^p (1 -\mathcal{L}_{\theta})^{1 / 2} \varphi^{m,
      \sharp} (t) \|\\
      \lesssim & (t e^{t C} + 1)^{1 / 2} (\| (1 +\mathcal{N})^p (1
      -\mathcal{L}_{\theta}) \varphi_0^{m, \sharp} \| + \| (1 +\mathcal{N})^{p
      + \alpha (0)} (1 -\mathcal{L}_{\theta})^{1 / 2} \varphi_0^{m, \sharp}
      \|) .
    \end{array} \label{eq:est-phi-sharp}
  \end{equation}
\end{lemma}

\begin{proof}
  \tmcolor{red}{}By~{\eqref{eq:ctrl-phi}} and Lemma~\ref{lem:exist-controlled}
  it follows that
  \[ \begin{array}{ll}
       & \| (1 +\mathcal{N})^p (1 -\mathcal{L}_{\theta})^{1 / 2} \varphi^{m,
       \sharp} (t) \|\\
       \lesssim & \| (1 +\mathcal{N})^p (1 -\mathcal{L}_{\theta})^{1 / 2}
       \varphi^m (t) \| + \| (1 +\mathcal{N})^p (1 -\mathcal{L}_{\theta})^{- 1
       / 2} \mathcal{G}^{m, \succ} \varphi^m (t) \|\\
       \lesssim & t e^{t C} \| (1 +\mathcal{N})^p \mathcal{L}^m \varphi^m_0 \|
       + \| (1 +\mathcal{N})^p (1 -\mathcal{L}_{\theta})^{1 / 2} \varphi^m_0
       \|\\
       \lesssim & (t e^{t C} + 1)^{1 / 2} (\| (1 +\mathcal{N})^p (1
       -\mathcal{L}_{\theta}) \varphi_0^{m, \sharp} \| + \| (1
       +\mathcal{N})^{p + \alpha (0)} (1 -\mathcal{L}_{\theta})^{1 / 2}
       \varphi_0^{m, \sharp} \|),
     \end{array} \]
  where in the last step we exploited Proposition~\ref{prop:well-def-gen}.
\end{proof}

For $\gamma \in] 1 / 2, 1 - 1 / (2 \theta) [$, we have that, by the
estimates~{\eqref{eq:unifbdG1}} and~{\eqref{eq:unifbdG2}},
\begin{eqnarray*}
  \| (1 +\mathcal{N})^p (1 -\mathcal{L}_{\theta})^{\gamma - 1} \mathcal{G}^{m,
  \succ} (1 -\mathcal{L}_{\theta}) \varphi^{m, \sharp} (s) \| & \lesssim & \|
  (1 +\mathcal{N})^{p + 3 / 2} (1 -\mathcal{L}_{\theta})^{\gamma + 1 / 2 + 1 /
  (2 \theta)} \varphi^{m, \sharp} (s) \|\\
  & \lesssim & \| (1 +\mathcal{N})^{p + 3 / 2} (1
  -\mathcal{L}_{\theta})^{\gamma + 1 / 2 + 1 / (2 \theta)} \varphi^{m, \sharp}
  (s) \| .
\end{eqnarray*}
By interpolation for products, there exists $q > 0$ such that, for all
$\varepsilon \in] 0, 1 [$,
\[ \begin{array}{lll}
     \| (1 +\mathcal{N})^{p + 3 / 2} (1 -\mathcal{L}_{\theta})^{\gamma + 1 / 2
     + 1 / (2 \theta)} \varphi^{m, \sharp} (s) \| & \lesssim & C_{\varepsilon}
     \| (1 +\mathcal{N})^q (1 -\mathcal{L}_{\theta})^{1 / 2} \varphi^{m,
     \sharp} (s) \|\\
     &  & + \varepsilon \| (1 +\mathcal{N})^p (1
     -\mathcal{L}_{\theta})^{\gamma + 1} \varphi^{m, \sharp} (s) \|,
   \end{array} \]
where the first term on the right-hand side can be controlled via the a priori
estimate~{\eqref{eq:est-phi-sharp}}, while the second term can be absorbed on
the left-hand side. Moreover, we have by~{\eqref{eq:est-G-minor}} and by
estimate~{\eqref{eq:est-phi-sharp}},
\[ \begin{array}{lll}
     \| (1 +\mathcal{N})^p (1 -\mathcal{L}_{\theta})^{\gamma} \mathcal{G}^{m,
     \prec} \varphi^m (s) \| & \lesssim & \| (1 +\mathcal{N})^{p + \alpha
     (\gamma)} (1 -\mathcal{L}_{\theta})^{1 / 2} \varphi^{m, \sharp} (s) \|\\
     & \lesssim & \| (1 +\mathcal{N})^{p + \alpha (\gamma)} (1
     -\mathcal{L}_{\theta}) \varphi^{m, \sharp}_0 \|\\
     &  & + \| (1 +\mathcal{N})^{p + \alpha (\gamma) + \alpha (0)} (1
     -\mathcal{L}_{\theta})^{1 / 2} \varphi^{m, \sharp}_0 \| .
   \end{array} \]
Recalling $\gamma \in] 1 / 2, 1 - 1 / (2 \theta) [$ and exploiting
estimates~{\eqref{eq:unifbdG1}}--{\eqref{eq:unifbdG2}}, we get
\[ \begin{array}{ll}
     & \| (1 +\mathcal{N})^p (1 -\mathcal{L}_{\theta})^{\gamma - 1}
     \mathcal{G}^{m, \succ} \mathcal{G}^{m, \prec} \varphi^m (s) \|\\
     \lesssim & \| (1 +\mathcal{N})^{p + 3 / 2} (1
     -\mathcal{L}_{\theta})^{\gamma - 1 / 2 + 1 / (2 \theta)} \mathcal{G}^{m,
     \prec} \varphi^m (s) \|\\
     \lesssim & \| (1 +\mathcal{N})^{p + 3 / 2 + \alpha (\gamma - 1 / 2 + 1 /
     (2 \theta))} (1 -\mathcal{L}_{\theta})^{1 / 2} \varphi^{m, \sharp} (s)
     \|\\
     \lesssim & \| (1 +\mathcal{N})^{p + \alpha (\gamma)} (1
     -\mathcal{L}_{\theta})^{1 / 2} \varphi^{m, \sharp} (s) \|,
   \end{array} \]
where we used $3 / 2 + \alpha (\gamma - 1 / 2 + 1 / (2 \theta)) < \alpha
(\gamma)$ whenever $\bar{\varepsilon} < 1 / 3 - 1 / (3 \theta)$. This bound
can be controlled via~{\eqref{eq:est-phi-sharp}} as above. As a consequence,
we established that, after renaming $q = q (p, \gamma) > 0$,
\[ \begin{array}{lll}
     \sup_{0 \leqslant t \leqslant T} \| (1 +\mathcal{N})^p (1
     -\mathcal{L}_{\theta})^{\gamma} \Phi^{m, \sharp} (t) \| & \lesssim_T & \|
     (1 +\mathcal{N})^q (1 -\mathcal{L}_{\theta}) \varphi^{m, \sharp}_0 \|\\
     &  & + \| (1 +\mathcal{N})^{q + \alpha (0)} (1 -\mathcal{L}_{\theta})^{1
     / 2} \varphi^{m, \sharp}_0 \|\\
     &  & + \varepsilon \sup_{0 \leqslant t \leqslant T} \| (1
     +\mathcal{N})^p (1 -\mathcal{L}_{\theta})^{\gamma + 1} \varphi^{m,
     \sharp} (t) \|,
   \end{array} \]
and hence, for $\gamma \in] 1 / 2, 1 - 1 / (2 \theta) [$,
\[ \begin{array}{lll}
     \sup_{0 \leqslant t \leqslant T} \| (1 +\mathcal{N})^p (1
     -\mathcal{L}_{\theta})^{1 + \gamma} \varphi^{m, \sharp} (t) \| &
     \lesssim_T & \| (1 +\mathcal{N})^q (1 -\mathcal{L}_{\theta}) \varphi^{m,
     \sharp}_0 \| + \| (1 +\mathcal{N})^q (1 -\mathcal{L}_{\theta})^{1 / 2}
     \varphi^{m, \sharp}_0 \|\\
     &  & + \| (1 +\mathcal{N})^p (1 -\mathcal{L}_{\theta})^{1 + \gamma}
     \varphi^{m, \sharp}_0 \|\\
     & \lesssim_T & \| (1 +\mathcal{N})^q (1 -\mathcal{L}_{\theta})^{1 +
     \gamma} \varphi^{m, \sharp}_0 \| .
   \end{array} \]
Recall $\partial_t \varphi^{m, \sharp} = (1 -\mathcal{L}_{\theta}) \varphi^{m,
\sharp} + \Phi^{m, \sharp} (t)$, so that
\[ \sup_{0 \leqslant t \leqslant T} \| (1 +\mathcal{N})^p (1
   -\mathcal{L}_{\theta})^{\gamma} \partial_t \varphi^{m, \sharp} (t) \|
   \lesssim \| (1 +\mathcal{N})^q (1 -\mathcal{L}_{\theta})^{1 + \gamma}
   \varphi^{m, \sharp}_0 \| . \]
By interpolation, this gives
\[ \| (1 +\mathcal{N})^p (1 -\mathcal{L}_{\theta})^{1 + \gamma / 2}
   (\varphi^{m, \sharp} (t) - \varphi^{m, \sharp} (s)) \| \leqslant | t - s
   |^{\kappa} \| (1 +\mathcal{N})^q (1 -\mathcal{L}_{\theta})^{1 + \gamma}
   \varphi^{m, \sharp}_0 \| . \]
Introduce now, for $p > 0$, the sets
\[ \mathcal{U}_p = \bigcup_{\gamma \in \left] \frac{1}{2}, 1 - \frac{1}{2
   \theta} \right[} \mathcal{K} (1 +\mathcal{N})^{q (p, \gamma)} (1
   -\mathcal{L}_{\theta})^{- 1 - \gamma} \mathcal{H} \subset \mathcal{H}, \]
and $\mathcal{U}= \bigcup_{p > \alpha (0)} \mathcal{U}_p$.

\begin{theorem}
  \label{thm:exist-kolmogorov}Let $p > 0$ and $\varphi_0 \in \mathcal{U}_p$.
  Then there exists a solution
  \[ \varphi \in \bigcup_{\delta > 0} C (\mathbb{R}_+ ; (1 +\mathcal{N})^{- p
     + \delta} (1 -\mathcal{L}_{\theta})^{- 1} \mathcal{H}) \]
  to the Kolmogorov backward equation $\partial_t \varphi =\mathcal{L}
  \varphi$ with initial condition $\varphi (0) = \varphi_0$. For $p > \alpha
  (0)$, we have $\varphi \in C (\mathbb{R}_+, \mathcal{D} (\mathcal{L})) \cap
  C^1 (\mathbb{R}, \mathcal{H})$ and, by dissipativity of~$\mathcal{L}$, this
  solution is unique.
\end{theorem}

\begin{proof}
  Let $\varphi_0 \in \mathcal{U}_p$ and set $\varphi_0^{\sharp} \assign
  \mathcal{K}^{- 1} \varphi_0 \in (1 +\mathcal{N})^{- q} (1
  -\mathcal{L}_{\theta})^{- 1 - \gamma} \mathcal{H}$ for $\gamma \in] 1 / 2, 1
  - 1 / (2 \theta) [$ and $p > 0$. For $m \in \mathbb{N}$, let $\varphi^m$ be
  the solution to $\partial_t \varphi^m =\mathcal{L}^m \varphi^m$ with initial
  condition $\varphi^m (0) =\mathcal{K}^m \varphi_0^{\sharp}$. A diagonal
  argument yields the relative compactness of bounded sets of $(1
  +\mathcal{N})^{- p} (1 -\mathcal{L}_{\theta})^{- 1 - \gamma / 2}
  \mathcal{H}$ in the space $(1 +\mathcal{N})^{- p + \delta} (1
  -\mathcal{L}_{\theta})^{- 1} \mathcal{H}$ for $\delta > 0$, with the
  consequence that, by Ascoli-Arzel{\`a} the sequence $(\varphi^{m,
  \sharp})_m$ is relatively compact in $C (\mathbb{R}_+ ; (1 +\mathcal{N})^{-
  p + \delta} (1 -\mathcal{L}_{\theta})^{- 1} \mathcal{H})$ equipped with the
  topology of uniform convergence on compact sets. We denote
  $\varphi^{\sharp}$ a limit point of such a sequence and let $\varphi
  =\mathcal{K} \varphi^{\sharp}$. Then, along the convergent subsequence,
  \begin{eqnarray*}
    \varphi (t) - \varphi (0) & = & \lim_{m \rightarrow \infty} (\varphi^m (t)
    - \varphi^m (0))\\
    & = & \lim_{m \rightarrow \infty} \int_0^t \mathcal{L}^m \varphi^m (s)
    \mathd s\\
    & = & \lim_{m \rightarrow \infty} \int_0^t (\mathcal{L}_{\theta}
    \varphi^{m, \sharp} (s) +\mathcal{G}^{m, \prec} \mathcal{K}^m \varphi^{m,
    \sharp} (s)) \mathd s\\
    & = & \lim_{m \rightarrow \infty} \int_0^t (\mathcal{L}_{\theta}
    \varphi^{\sharp} (s) +\mathcal{G}^{m, \prec} \mathcal{K}^m \varphi^{m,
    \sharp} (s)) \mathd s\\
    & = & \int_0^t (\mathcal{L}_{\theta} \varphi^{\sharp} (s)
    +\mathcal{G}^{\prec} \mathcal{K} \varphi^{\sharp} (s)) \mathd s,
  \end{eqnarray*}
  where we exploited our uniform bounds on $\mathcal{L}_{\theta}$,
  $\mathcal{G}^{m, \succ}$, $\mathcal{K}^m$ and the convergence of
  $\varphi^{m, \sharp}$ to $\varphi^{\sharp}$ as $m \rightarrow \infty$ to get
  the 4th equality, while the last step follows from our bounds for
  $\mathcal{G}^{\prec}$ and $\mathcal{K}$, together with the dominated
  convergence theorem.
  
  If we take $p > \alpha (0)$, then by definition
  (cfr.~Lemma~\ref{lem:density}) $\varphi \in \mathcal{D} (\mathcal{L})$.
  Furthermore, $\mathcal{L} \varphi \in C (\mathbb{R}_+ ; \mathcal{H})$ and we
  have $\varphi \in C^1 (\mathbb{R}_+ ; \mathcal{H})$ because of the relation
  $\varphi (t) - \varphi (s) = \int_s^t \mathcal{L} \varphi (\tau) \mathd
  \tau$. We can hence compute,
  \[ \partial_t \| \varphi (t) \|^2 = 2 \langle \varphi (t), \mathcal{L}
     \varphi (t) \rangle \leqslant 0, \]
  by the dissipativity of the operator $\mathcal{L}$ given by
  Lemma~\ref{lem:diss}. Therefore, for any solution we have $\| \varphi (t) \|
  \leqslant \| \varphi_0 \|$, which together with the linearity of the
  equation yields the uniqueness.
\end{proof}

\section{Bounds on the drift}\label{sec:bounds}

We prove there the key bounds on the drift $\mathcal{G}^m$.

\begin{proof*}{Proof of Lemma~\ref{lem:uniform-bds-G}}
  We start by estimating $\mathcal{G}_+^m$. We have, by
  Lemma~\ref{lem:app-estimate} and since $\gamma > 1 / (2 \theta)$,

  \begin{equation}
    \begin{array}{ll}
      & \| w (\mathcal{N}) (1 -\mathcal{L}_{\theta})^{- \gamma}
      \mathcal{G}_+^m \varphi \|^2 = \sum_{n \geqslant 0} n! w (n)^2
      \sum_{k_{1 : n}} \left( \prod_{i = 1}^n | 2 \pi k_i |^2 \right) |
      \mathcal{F} ((1 -\mathcal{L}_{\theta})^{- \gamma} \mathcal{G}_+^m
      \varphi)_n (k_{1 : n}) |^2\\
      \lesssim & \sum_{n \geqslant 2} n! n^2 w (n)^2 \sum_{k_{1 : n}} \left(
      \prod_{i = 1}^n | 2 \pi k_i |^2 \right) \frac{\mathbbm{1}_{| k_1 |, |
      k_2 |, | k_1 + k_2 | \leqslant m} | k_1 + k_2 |^4}{(1 + L_{\theta} (k_{1
      : n}))^{2 \gamma} | k_1 |^2 | k_2 |^2} | \hat{\varphi}_{n - 1} (k_1 +
      k_2, k_{3 : n}) |^2\\
      \lesssim & \sum_{n \geqslant 2} n! n^2 w (n)^2 \sum_{\ell, k_{3 : n}}
      \left( \prod_{i = 3}^n | 2 \pi k_i |^2 \right) | \ell |^4 |
      \hat{\varphi}_{n - 1} (\ell, k_{3 : n}) |^2 \sum_{k_1 + k_2 = \ell} (1 +
      L_{\theta} (k_{1 : n}))^{- 2 \gamma}\\
      \lesssim & \sum_{n \geqslant 2} n! n^2 w (n)^2 \sum_{\ell, k_{3 : n}}
      \left( \prod_{i = 3}^n | 2 \pi k_i |^2 \right) | \ell |^4 (1 +
      L_{\theta} (\ell, k_{3 : n}))^{- 2 \gamma + 1 / \theta} |
      \hat{\varphi}_{n - 1} (\ell, k_{3 : n}) |^2 .
    \end{array} \label{eq:ineq1-proofbds}
  \end{equation}
  Introducing the notation $\ell_1 = \ell = k_1 + k_2$ and $\ell_i = k_{i +
  1}$ for $i \geqslant 2$, we get
  \[ \begin{array}{ll}
       \lesssim & \sum_{n \geqslant 2} n! n^2 w (n)^2 \sum_{\ell_{1 : n - 1}}
       \left( \prod_{i = 1}^{n - 1} | 2 \pi \ell_i |^2 \right) | \ell_1 |^2 (1
       + L_{\theta} (\ell_{1 : n - 1}))^{- 2 \gamma + 1 / \theta} |
       \hat{\varphi}_{n - 1} (\ell_{1 : n - 1}) |^2 .
     \end{array} \]
  then using the symmetry of $\hat{\varphi}_{n - 1}$ we reduce this to
  \[ \begin{array}{ll}
       \lesssim & \sum_{n \geqslant 2} n! n w (n)^2 \sum_{\ell_{1 : n - 1}}
       \left( \prod_{i = 1}^{n - 1} | 2 \pi \ell_i |^2 \right) \frac{| \ell_1
       |^2 + \cdots + | \ell_n |^2}{1 + L_{\theta} (\ell_{1 : n - 1})} (1 +
       L_{\theta} (\ell_{1 : n - 1}))^{- 2 \gamma + 1 / \theta + 1} |
       \hat{\varphi}_{n - 1} (\ell_{1 : n - 1}) |^2 .
     \end{array} \]
  from which we obtain
  \[ \lesssim \sum_{n \geqslant 1} n! (n + 1)^2 w (n + 1)^2 \sum_{\ell_{1 :
     n}} \left( \prod_{i = 1}^n | 2 \pi \ell_i |^2 \right) (1 + L_{\theta}
     (\ell_{1 : n}))^{- 2 \gamma + 1 / \theta + 1} | \hat{\varphi}_n (\ell_{1
     : n}) |^2 \]
  \[ \lesssim \| w (\mathcal{N}+ 1) (1 +\mathcal{N}) (1
     -\mathcal{L}_{\theta})^{(1 + 1 / \theta) / 2 - \gamma} \varphi \|^2 . \]
  For $\mathcal{G}_-^m$, note first that, by the Cauchy-Schwarz inequality and
  by Lemma~\ref{lem:app-estimate} (since $\gamma < 1 / 2$),
  \[ \begin{array}{ll}
       & \left| \sum_{p + q = k_1} (k_1^{\perp} \cdot p)  (k_1 \cdot q)
       \hat{\varphi}_{n + 1} (p, q, k_{2 : n}) \right|^2\\
       \lesssim & \sum_{p + q = k_1} (1 + | p |^{2 \theta} + | q |^{2
       \theta})^{2 \gamma - 1 - 1 / \theta}\\
       & \times \sum_{p + q = k_1} (1 + | p |^{2 \theta} + | q |^{2
       \theta})^{1 + 1 / \theta - 2 \gamma} | k_1^{\perp} \cdot p |^2  | k_1
       \cdot q |^2 | \hat{\varphi}_{n + 1} (p, q, k_{2 : n}) |^2\\
       \lesssim & (1 + | k_1 |^{2 \theta})^{2 \gamma - 1} \sum_{p + q = k_1}
       (1 + | p |^{2 \theta} + | q |^{2 \theta})^{1 + 1 / \theta - 2 \gamma} |
       k_1^{\perp} \cdot p |^2  | k_1 \cdot q |^2 | \hat{\varphi}_{n + 1} (p,
       q, k_{2 : n}) |^2,
     \end{array} \]
  therefore,
  \[ \begin{array}{ll}
       & \| w (\mathcal{N}) (1 -\mathcal{L}_{\theta})^{- \gamma}
       \mathcal{G}_-^m \varphi \|^2 = \sum_{n \geqslant 0} n! w (n)^2
       \sum_{k_{1 : n}} \left( \prod_{i = 1}^n | 2 \pi k_i |^2 \right) |
       \mathcal{F} ((1 -\mathcal{L}_{\theta})^{- \gamma} \mathcal{G}_-^m
       \varphi)_n (k_{1 : n}) |^2\\
       \lesssim & \sum_{n \geqslant 0} n! w (n)^2 (n + 1)^4 \sum_{k_{1 : n}}
       \left( \prod_{i = 1}^n | 2 \pi k_i |^2 \right) \frac{1}{| k_1 |^4 (1 +
       L_{\theta} (k_{1 : n}))^{2 \gamma}}\\
       & \times \left| \sum_{p + q = k_1} (k_1^{\perp} \cdot p)  (k_1 \cdot
       q) \hat{\varphi}_{n + 1} (p, q, k_{2 : n}) \right|^2\\
       \lesssim & \sum_{n \geqslant 0} n! w (n)^2 (n + 1)^4 \sum_{k_{1 : n}}
       \left( \prod_{i = 1}^n | 2 \pi k_i |^2 \right) \frac{1}{| k_1 |^4 (1 +
       L_{\theta} (k_{1 : n}))^{2 \gamma}} (1 + | k_1 |^{2 \theta})^{2 \gamma
       - 1}\\
       & \times \sum_{p + q = k_1} (1 + | p |^{2 \theta} + | q |^{2
       \theta})^{1 + 1 / \theta - 2 \gamma} | k_1 |^4 | p |^2 | q |^2 |
       \hat{\varphi}_{n + 1} (p, q, k_{2 : n}) |^2\\
       \lesssim & \sum_{n \geqslant 0} n! w (n)^2 (n + 1)^4 \sum_{k_{1 : n}}
       \sum_{p + q = k_1} \left( \prod_{i = 2}^n | 2 \pi k_i |^2 \right) | 2
       \pi p |^2 | 2 \pi q |^2\\
       & \times (1 + | p |^{2 \theta} + | q |^{2 \theta})^{1 + 1 / \theta - 2
       \gamma} | \hat{\varphi}_{n + 1} (p, q, k_{2 : n}) |^2,
     \end{array} \]
  we now let $\ell_1 = p$, $\ell_2 = q$, and $\ell_i = k_{i - 1}$ for $3
  \leqslant i \leqslant n + 1$, so that
  \[ \begin{array}{ll}
       & \| w (\mathcal{N}) (1 -\mathcal{L}_{\theta})^{- \gamma}
       \mathcal{G}_-^m \varphi \|^2\\
       \lesssim & \sum_{n \geqslant 0} n! w (n)^2 (n + 1)^4 \sum_{\ell_{1 : n
       + 1}} \left( \prod_{i = 1}^{n + 1} | 2 \pi \ell_i |^2 \right) (1 + |
       \ell_1 |^{2 \theta} + | \ell_2 |^{2 \theta})^{1 + 1 / \theta - 2
       \gamma} | \hat{\varphi}_{n + 1} (\ell_{1 : n + 1}) |^2\\
       \lesssim & \sum_{n \geqslant 0} n! w (n)^2 (n + 1)^4 \sum_{\ell_{1 : n
       + 1}} \left( \prod_{i = 1}^{n + 1} | 2 \pi \ell_i |^2 \right) (1 + |
       \ell_1 |^{2 \theta} + \cdots + | \ell_{n + 1} |^{2 \theta})^{1 + 1 /
       \theta - 2 \gamma} | \hat{\varphi}_{n + 1} (\ell_{1 : n + 1}) |^2\\
       \lesssim & \sum_{n \geqslant 1} n! w (n - 1)^2 n^3 \sum_{\ell_{1 : n}}
       \left( \prod_{i = 1}^n | 2 \pi \ell_i |^2 \right) (1 + | \ell_1 |^{2
       \theta} + \cdots + | \ell_n |^{2 \theta})^{1 + 1 / \theta - 2 \gamma} |
       \hat{\varphi}_n (\ell_{1 : n}) |^2\\
       \lesssim & \| w (\mathcal{N}- 1) \mathcal{N}^{3 / 2} (1
       -\mathcal{L}_{\theta})^{(1 + 1 / \theta) / 2 - \gamma} \varphi \|^2
     \end{array} \]
  which gives the uniform bound.
  
  \
  
  Let us now discuss the $m$-dependent estimates, we have for
  $\mathcal{G}_+^m$
  \[ \begin{array}{ll}
       & \| w (\mathcal{N}) \mathcal{G}_+^m \varphi \|^2 = \sum_{n \geqslant
       0} n! w (n)^2 \sum_{k_{1 : n}} \left( \prod_{i = 1}^n | 2 \pi k_i |^2
       \right) | \mathcal{F} (\mathcal{G}_+^m \varphi)_n (k_{1 : n}) |^2\\
       \lesssim & \sum_{n \geqslant 2} n! w (n)^2 n^2 \sum_{k_{1 : n}} \left(
       \prod_{i = 1}^n | 2 \pi k_i |^2 \right) \mathbbm{1}_{| k_1 |, | k_2 |,
       | k_1 + k_2 | \leqslant m} \frac{| k_1 + k_2 |^4}{| k_1 |^2 | k_2 |^2}
       | \hat{\varphi}_{n - 1} (k_1 + k_2, k_{3 : n}) |^2\\
       \lesssim & \sum_{n \geqslant 2} n! w (n)^2 n^2\\
       & \times \sum_{k_{1 : n}} \left( \prod_{i = 3}^n | 2 \pi k_i |^2
       \right) | 2 \pi (k_1 + k_2) |^2 \mathbbm{1}_{| k_1 |, | k_2 |, | k_1 +
       k_2 | \leqslant m} | k_1 + k_2 |^{2 \theta} | \hat{\varphi}_{n - 1}
       (k_1 + k_2, k_{3 : n}) |^2\\
       \lesssim & m^2 \sum_{n \geqslant 2} n! w (n)^2 n^2 \sum_{\ell_{1 : n -
       1}} \left( \prod_{i = 1}^{n - 1} | 2 \pi \ell_i |^2 \right) | \ell_1
       |^{2 \theta} | \hat{\varphi}_{n - 1} (\ell_{1 : n - 1}) |^2\\
       \lesssim & m^2 \sum_{n \geqslant 2} n! w (n)^2 n \sum_{\ell_{1 : n -
       1}} \left( \prod_{i = 1}^{n - 1} | 2 \pi \ell_i |^2 \right) L_{\theta}
       (\ell_{1 : n - 1}) | \hat{\varphi}_{n - 1} (\ell_{1 : n - 1}) |^2\\
       \lesssim & m^2 \sum_{n \geqslant 1} n! w (n + 1)^2 (n + 1)^2
       \sum_{\ell_{1 : n}} \left( \prod_{i = 1}^n | 2 \pi \ell_i |^2 \right)
       (1 + L_{\theta} (\ell_{1 : n})) | \hat{\varphi}_n (\ell_{1 : n}) |^2\\
       \lesssim & m^2 \| w (\mathcal{N}+ 1) (1 +\mathcal{N}) (1
       -\mathcal{L}_{\theta})^{1 / 2} \varphi \|^2 .
     \end{array} \]
  Finally, for $\mathcal{G}_-^m$ we have,
  \[ \begin{array}{ll}
       & \| w (\mathcal{N}) \mathcal{G}_-^m \varphi \|^2 = \sum_{n \geqslant
       0} n! w (n)^2 \sum_{k_{1 : n}} \left( \prod_{i = 1}^n | 2 \pi k_i |^2
       \right) | \mathcal{F} (\mathcal{G}_-^m \varphi)_n (k_{1 : n}) |^2\\
       \lesssim & \sum_{n \geqslant 0} n! w (n)^2 (n + 1)^4 \sum_{k_{1 : n}}
       \left( \prod_{i = 1}^n | 2 \pi k_i |^2 \right) \frac{\mathbbm{1}_{| k_1
       |, | p |, | q | \leqslant m}}{| k_1 |^4}\\
       & \times \left| \sum_{p + q = k_1} (k_1^{\perp} \cdot p)  (k_1 \cdot
       q) \hat{\varphi}_{n + 1} (p, q, k_{2 : n}) \right|^2\\
       \lesssim & \sum_{n \geqslant 0} n! w (n)^2 (n + 1)^4\\
       & \times \sum_{k_{1 : n}} \sum_{p + q = k_1} \left( \prod_{i = 2}^n |
       2 \pi k_i |^2 \right) | 2 \pi p |^2 | 2 \pi q |^2 \mathbbm{1}_{| k_1 |,
       | p |, | q | \leqslant m} | k_1 |^2 | \hat{\varphi}_{n + 1} (p, q, k_{2
       : n}) |^2\\
       \lesssim & m^2 \sum_{n \geqslant 0} n! w (n)^2 (n + 1)^3 \sum_{p, q,
       k_{2 : n}} \left( \prod_{i = 2}^n | 2 \pi k_i |^2 \right)\\
       & \times | 2 \pi p |^2 | 2 \pi q |^2 (| p |^{2 \theta} + | q |^{2
       \theta} + | k_2 |^{2 \theta} + \cdots + | k_n |^{2 \theta}) |
       \hat{\varphi}_{n + 1} (p, q, k_{2 : n}) |^2\\
       \lesssim & m^2 \sum_{n \geqslant 0} n! w (n)^2 (n + 1)^3 \sum_{\ell_{1
       : n + 1}} \left( \prod_{i = 1}^{n + 1} | 2 \pi \ell_i |^2 \right)
       L_{\theta} (\ell_{1 : n + 1}) | \hat{\varphi}_{n + 1} (\ell_{1 : n +
       1}) |^2\\
       \lesssim & m^2 \sum_{n \geqslant 1} n! w (n - 1)^2 n^2 \sum_{\ell_{1 :
       n}} \left( \prod_{i = 1}^n | 2 \pi \ell_i |^2 \right) L_{\theta}
       (\ell_{1 : n}) | \hat{\varphi}_n (\ell_{1 : n}) |^2\\
       \lesssim & m^2 \| w (\mathcal{N}- 1) \mathcal{N} (1
       -\mathcal{L}_{\theta})^{1 / 2} \varphi \|^2 .
     \end{array} \]
  
\end{proof*}

\section{Stochastic Navier--Stokes on the plane}\label{s:whole-space}

In this section we prove that the main results of the paper, namely existence
and uniqueness of energy solution for the hyper-viscous Navier--Stokes
eq.~{\eqref{eq:navierstokes}} extends very naturally to the setting of the
whole plane $\mathbb{R}^2$. We will discuss first existence of martingale
solutions via a limiting procedure involving finite volume approximations,
then we will show that the Kolmogorov equation can be solved also in the full
space, which implies, as in the periodic setting, uniqueness in law.

\

\paragraph{The invariant measure}At the beginning of the paper, we introduced
the invariant measure of the problem, i.e., the energy measure~$\mu$ given by
eq.~{\eqref{eq:energy-measure}}. The computation of the covariance in the
current case gives, for any $\varphi, \psi \in \mathcal{S}$,
\[ \mathbb{E} [\omega (\varphi) \omega (\psi)] = \langle (- \Delta)^{1 / 2}
   \varphi, (- \Delta)^{1 / 2} \psi \rangle_{L^2 (\mathbb{R}^2)} \backassign
   \langle \varphi, \psi \rangle_{\dot{H}^1 (\mathbb{R}^2)}, \]
where $\dot{H}^s (\mathbb{R}^d)$ denotes the so-called {\tmem{homogeneous
Sobolev space}} of $L^2 (\mathbb{R}^d)$ functions $f$ having norm
$\int_{\mathbb{R}^d} | \xi |^{2 s} | \hat{f} (\xi) |^2 \mathd \xi$ finite. We
denote by $\mu_{\mathbb{R}^2}$ the energy measure in the case of the whole
space.

\paragraph{A new approximating problem}In order to approximate stochastic
Navier-Stokes equations on the whole space, we study Galerkin approximation
problems on scaled tori $\mathbb{T}^2_{\lambda} \assign \mathbb{R}^2 \setminus
(2 \pi \lambda \mathbb{Z}^2)$, $\lambda > 0$, with the goal to take first the
limit as~$\lambda \rightarrow \infty$, allowing us to pass to the case
of~$\mathbb{R}^2$. For $f : \mathbb{T}^2_{\lambda} \rightarrow \mathbb{R}$, we
define the Fourier transform $\mathcal{F}_{\lambda} (f) = \invbreve{f} :
\lambda^{- 1} \mathbb{Z}^2 \rightarrow \mathbb{R}$ as
\[ \mathcal{F}_{\lambda} (f) (k) = \invbreve{f} (k) =
   \int_{\mathbb{T}_{\lambda}^2} e^{- 2 \pi \iota k \cdot x} f (x) \mathd x,
   \qquad k \in \lambda^{- 1} \mathbb{Z}^2, \]
while the inverse transform is given by
\[ \mathcal{F}_{\lambda}^{- 1} (f) (x) = (2 \pi \lambda)^{- 2} \sum_{k \in
   \lambda^{- 1} \mathbb{Z}^2} e^{2 \pi \iota k \cdot x} f (k), \qquad x \in
   \mathbb{T}^2_{\lambda} . \]
Plancherel theorem now reads as
\[ (2 \pi \lambda)^{- 2} \sum_{k \in \lambda^{- 1} \mathbb{Z}^2}
   \mathcal{F}_{\lambda} (f) (k)  \overline{\mathcal{F}_{\lambda} (g) (k)_{}}
   = \int_{\mathbb{T}_{\lambda}^2} f (x) g (x) \mathd x. \]
The $\mathcal{H}$-norm is now given by
\[ \| \varphi \|_{\mathcal{H}_{\lambda}} = \sum_{n = 0}^{\infty} n! \|
   \varphi_n \|^2_{(H^1_0 (\mathbb{T}_{\lambda}^2))^{\otimes n}} \simeq
   \sum_{n = 0}^{\infty} n! \lambda^{- 2 n} \sum_{k_{1 : n} \in (\lambda^{- 1}
   \mathbb{Z}^2_0)^n} \left( \prod_{i = 1}^n | k_i |^2  \right)  |
   \invbreve{\varphi}_n (k_{1 : n}) |^2 . \]
The Biot-Savart kernel is $K (x) = - (2 \pi)^{- 3} \lambda^{- 2} \iota \sum_{k
\in \lambda^{- 1} \mathbb{Z}^2} k^{\bot} | k |^{- 2} e^{2 \pi \iota k \cdot
x}$, for $x \in \mathbb{T}^2_{\lambda}$, since from the relation $\omega =
\nabla^{\bot} \cdot u$ we get
\[ \invbreve{\omega} (k) = 2 \pi \iota k_2  \invbreve{u}_1 (k) - 2 \pi \iota
   k_1  \invbreve{u}_2 (k) = 2 \pi \iota k^{\bot} \cdot \invbreve{u} (k), \]
which gives $\invbreve{u} (k) = - 2 \pi \iota k^{\top} | 2 \pi k |^{- 2} \cdot
\invbreve{\omega} (k)$.

$\mathcal{L}_{\theta}^{\lambda, m}$ can again be represented in Fourier terms
by~{\eqref{eq:ltheta-fourier}} for $k_1, \ldots, k_n \in \lambda^{- 1}
\mathbb{Z}^2$. The ($\lambda$-)Fourier transform of $\mathcal{G}_+^{\lambda,
m}$ is exactly the same as in {\eqref{eq:gp-fourier}}, while the one of
$\mathcal{G}_-^{\lambda, m}$ is as in~{\eqref{eq:gm-fourier}} but multiplied
by a factor $\lambda^{- 2}$ due to the convolution. Following the proof of
Lemma~\ref{lem:uniform-bds-G}, we get some estimates of
$\mathcal{G}_{\pm}^{\lambda, m}$ uniform both in $m$ and in $\lambda$ (up to
the $\lambda$-dependence of the $\mathcal{H}_{\lambda}$-norm). After getting
this estimate we obtain the same result as in
Lemma~\ref{lemma:galerkin-estimates}, for every $\lambda > 0$.

\begin{proof*}{Proof of Lemma~\ref{lem:uniform-bds-G}}
  We show here the two bounds for $\mathcal{G}_{\pm}^{\lambda, m}$. For
  $\mathcal{G}_+^{\lambda, m}$, the main difference with respect to the proof
  presented in Section~\ref{sec:bounds} is the presence of the $\lambda^{- 2
  n}$ term in the definition of the norm (and, of course, of a different
  Fourier transform), the sum on $k_1 + k_2 = \ell$ in the third step of the
  inequality~{\eqref{eq:ineq1-proofbds}} eats also a term $\lambda^{- 2}$,
  hence thereafter we will have $\lambda^{- 2 (n - 1)}$, which is the correct
  term that will enter the norm at the end of the estimate.
  
  For the term $\mathcal{G}_-^{\lambda, m}$ we will use the fact that, by
  Lemma~\ref{lem:app-estimate},
  \begin{eqnarray*}
    &  & \left| \lambda^{- 2} \sum_{p + q = k_1} (k_1^{\perp} \cdot p)  (k_1
    \cdot q) \invbreve{\varphi}_{n + 1} (p, q, k_{2 : n}) \right|^2\\
    & \lesssim & \lambda^{- 2} \sum_{p + q = k_1} (1 + | p |^{2 \theta} + | q
    |^{2 \theta})^{2 \gamma - 1 - 1 / \theta}\\
    &  & \times \lambda^{- 2} \sum_{p + q = k_1} (1 + | p |^{2 \theta} + | q
    |^{2 \theta})^{1 + 1 / \theta - 2 \gamma} | k_1^{\perp} \cdot p |^2  | k_1
    \cdot q |^2 | \invbreve{\varphi}_{n + 1} (p, q, k_{2 : n}) |^2\\
    & \lesssim & (1 + | k_1 |^{2 \theta})^{2 \gamma - 1} \lambda^{- 2}
    \sum_{p + q = k_1} (1 + | p |^{2 \theta} + | q |^{2 \theta})^{1 + 1 /
    \theta - 2 \gamma} | k_1^{\perp} \cdot p |^2  | k_1 \cdot q |^2 |
    \invbreve{\varphi}_{n + 1} (p, q, k_{2 : n}) |^2,
  \end{eqnarray*}
  which implies that $\| w (\mathcal{N}) (1 -\mathcal{L}_{\theta})^{- \gamma}
  \mathcal{G}_-^m \varphi \|^2$ can be bounded as in the proof in
  Section~\ref{sec:bounds} with the extra term $\lambda^{- 2 (n + 1)}$, that
  is exactly the term that will enter the norm, yielding the claimed estimate.
\end{proof*}

\paragraph{Existence for the cylinder martingale problem}We now want to give a
proper definition of infinitesimal generator and of martingale problem on the
whole space~$\mathbb{R}^2$ (cfr.~Section~\ref{sec:cyl-mart}), and in
particular we want to show the existence of solutions obtained in
Theorem~\ref{thm:existence-cyl-martingalepb} for the present scenario.

Let us therefore focus our attention on the proof of
Theorem~\ref{thm:existence-cyl-martingalepb}. Following Step~1, we want to
show tightness of the sequence $(\omega^{\lambda, m})_{\lambda, m}$. With
Step~2, we are going to conclude the existence for the martingale problem as
$\lambda, m \rightarrow \infty$. Let us assume for the moment that we can
associate to each cylinder function, $\varphi \in \tmop{Cyl}_{\mathbb{R}^2}$,
of the form $\varphi (\omega) = \Phi (\omega (f_1), \ldots, \omega (f_n))$,
where $f_1, \ldots, f_n \in \mathcal{S} (\mathbb{R}^2)$, to $\varphi^{\lambda}
(\omega) = \Phi (\omega (f_1^{\lambda}), \ldots, \omega (f_n^{\lambda}))$ with
$f_1^{\lambda}, \ldots, f^{\lambda}_n \in \mathcal{S}
(\mathbb{T}^2_{\lambda})$ in such a way that
\[ \varphi_{\ell}^{\lambda} (k_{1 : \ell}) = \varphi_{\ell} (k_{1 : \ell}),
   \qquad \text{for } k_1, \ldots, k_{\ell} \in \lambda^{- 1} \mathbb{Z}^2, \]
where we are exploiting the chaos decompositions of $\varphi$
and~$\varphi^{\lambda}$. It is then possible to recover the
bound~{\eqref{eq:bd-tightness}} for $\omega^{\lambda, m}$ and therefore
tightness, since $\mathcal{L}_{\theta}^{\lambda, m}$ has
$\mathcal{L}_{\theta}^{\infty}$ as a limit of a sum converging to an integral.
As regards Step~2, the crucial part is to pass from $\mathcal{L}^{\lambda, m}$
to $\mathcal{L}^{\infty, m}$ in~{\eqref{eq:limit-existence}} when taking the
limit as $\lambda \rightarrow \infty$. To do this, we have to show that $\| (1
-\mathcal{L}_{\theta})^{- 1 / 2} (\mathcal{L}^{\lambda, m} \varphi
-\mathcal{L}^{\infty, m} \varphi) \|_{\mathcal{H}_{\lambda}}$ tends to~$0$
when $\lambda \rightarrow \infty$, which reduces to prove that
\[ \| (1 -\mathcal{L}_{\theta})^{- 1 / 2} (\mathcal{G}_-^{\lambda, m} \varphi
   -\mathcal{G}_-^{\infty, m} \varphi) \|_{\mathcal{H}_{\lambda}} \rightarrow
   0, \qquad \text{as } \lambda \rightarrow \infty . \]
Comparing the explicit formulas for $\mathcal{G}_-^{\lambda, m}$ and
$\mathcal{G}_-^{\infty, m}$, we have that the only difference between the two
is the fact that in $\mathcal{G}^{\lambda, m}_-$ the sum given by the
convolution becomes an integral when taking the limit. In particular,
\begin{eqnarray*}
  &  & \| (1 -\mathcal{L}_{\theta})^{- 1 / 2} (\mathcal{G}_-^{\lambda, m}
  \varphi -\mathcal{G}_-^{\infty, m} \varphi) \|_{\mathcal{H}_{\lambda}}\\
  & \simeq & \sum_{n \geqslant 0} n! \lambda^{- 2 n} n (n + 1) \sum_{k_{1 :
  n} \in (\lambda^{- 1} \mathbb{Z}^2)^n} \left( \prod_{i = 1}^n | k_i |^2
  \right) \frac{1}{(1 + L_{\theta} (k_{1 : n}))^{1 / 2}} \times\\
  &  & \times \left| \lambda^{- 2} \sum_{p + q = k_1} \mathbbm{1}_{| k_1 |, |
  p |, | q | \leqslant m}  \frac{k_1^{\bot} \cdot p | q |^2}{| k_1 |^2}
  \mathcal{F}_{\lambda} (\varphi_{n + 1}) (p, q, k_{3 : n}) - \right. \\
  &  & \qquad \qquad \qquad \left. - \int_{\mathbb{R}^2} \mathbbm{1}_{| k_1
  |, | s |, | k_1 - s | \leqslant m}  \frac{k_1^{\bot} \cdot s | k_1 - s
  |^2}{| k_1 |^2} \mathcal{F}_{\lambda} (\varphi_{n + 1}) (s, k_1 - s, k_{3 :
  n}) \mathd s \right|^2 \\
  & \lesssim & \sum_{n \geqslant 0} n! \lambda^{- 2 n} n (n + 1) \sum_{k_{1 :
  n} \in (\lambda^{- 1} \mathbb{Z}^2)^n} \left( \prod_{i = 2}^n | k_i |^2
  \right)  \frac{\mathbbm{1}_{| k_1 | \leqslant m}}{(1 + L_{\theta} (k_{1 :
  n}))^{1 / 2}} \times\\
  &  & \times \left| \lambda^{- 2} \sum_p \mathbbm{1}_{| p |, | k_1 - p |
  \leqslant m} p | k_1 - p |^2 \mathcal{F}_{\lambda} (\varphi_{n + 1}) (p, k_1
  - p, k_{3 : n}) \right.,\\
  &  & \qquad \qquad \qquad \qquad \left. - \int_{\mathbb{R}^2}
  \mathbbm{1}_{| s |, | k_1 - s | \leqslant m} s | k_1 - s |^2
  \mathcal{F}_{\lambda} (\varphi_{n + 1}) (s, k_1 - s, k_{3 : n}) \mathd s
  \right|^2,
\end{eqnarray*}
and the right-hand side goes to zero as $\lambda \rightarrow \infty$.

In both cases it is important to understand the role played by the test
functions with respect to the norm (which depends on~$\lambda$) we are
considering. We want in fact to take $\varphi \in \tmop{Cyl}_{\mathbb{R}^2}$,
but we will evaluate it on $\omega^{m, \lambda}$. This is a non-trivial step
and it is worth to spend a few words on~it adopting a chaos expansion point of
view. To apply the Mitoma's criterion we only need to test on linear
functions, see~{\cite{Mitoma1983}}, while for the second step it suffices to
consider functions $\varphi$ of the form
\[ \varphi (\omega) = \of e^{\iota \langle \omega, f \rangle} : = e^{\iota
   \omega (f) + \frac{1}{2} \| f \|^2_{}}, \qquad \text{for } f \in
   \mathcal{S} (\mathbb{R}^2), \]
where $\of \nosymbol e^{\iota g} :$ indicates the Wick exponential of~$g$, see
also~{\cite{gubinelli2019random,gubinelli2015lectures}}. Focusing on the
latter case, we can identify $\varphi$ with the sequence of chaoses
$(\varphi_n (k_{1 : n}))_n$ where $\varphi_n (k_{1 : n}) = \iota^n \hat{f}
(k_1) \cdots \hat{f} (k_n)$ for $k_{1 : n} \in (\mathbb{R}^2)^n$, so that,
after noticing
\[ \varphi^{\lambda} (\omega^{m, \lambda}) \assign \varphi (\omega^{m,
   \lambda}) = e^{\iota \omega^{m, \lambda} (f) + \frac{1}{2} \| f
   \|^2_{\mathcal{S} (\mathbb{R}^2)}} = C_{\lambda} (f) e^{\iota
   \omega^{\lambda} (f) + \frac{1}{2} \| f \|^2_{\mathcal{S}
   (\mathbb{T}^2_{\lambda})}}, \]
with $C_{\lambda} (f) \rightarrow 1$ as $\lambda \rightarrow \infty$, we have
$\varphi^{\lambda}_n (k_{1 : n}) = \iota^n C_{\lambda} (f) \invbreve{f} (k_1)
\cdots \invbreve{f} (k_n)$ for $k_{1 : n} \in (\lambda^{- 1} \mathbb{Z}^2)^n$.

\paragraph{The Kolmogorov backward equation in the plane}Let us now turn to
the study of the Kolmogorov backward equation for the whole space setting. In
order to get the result we have obtained in Section~\ref{s:kolmogorov} in the
case of periodic boundary conditions, we need first to give a proper
description of the space we are working~in, that is
$\mathcal{H}_{\mathbb{R}^2} = L^2 (\mu_{\mathbb{R}^2})$. In particular, we
need that the operator $\mathcal{L}$ is well-defined on this space.

We start by describing the homogeneous Sobolev space~$\dot{H}^1
(\mathbb{R}^2)$. As remarked in Proposition~1.34
of~{\cite{bahouri2011fourier}} this space is not a space of functions. Indeed
consider a smooth bump function $\theta : \mathbb{R}^2 \rightarrow \mathbb{R}$
compactly supported and such that $\theta (0) = 1$. Let $\theta_{\varepsilon}
(x) = \theta (\varepsilon x)$, then as $\varepsilon \rightarrow 0$ we have $\|
\theta_{\varepsilon} \|_{\dot{H}^1 (\mathbb{R}^2)} \approx 1$,
$\theta_{\varepsilon} \rightarrowlim 1$ pointwise and $\theta_{\varepsilon}
\rightarrow 0$ weakly in $\dot{H}^1 (\mathbb{R}^2)$. The elements of
$\dot{H}^1 (\mathbb{R}^2)$ consists of equivalence classes of functions modulo
constants and we will have to take this into account in our analysis. We say
that $\varphi \in \dot{H}^1 (\mathbb{R}^2)$ if there exists a tempered
distribution $\tilde{\varphi} \in \mathcal{S}' (\mathbb{R}^2)$ such that
\[ \| \tilde{\varphi} \|_{\dot{H}^1 (\mathbb{R}^2)} = \int_{\mathbb{R}^2} | k
   |^2  | \widehat{\tilde{\varphi}} (k) |^2 \mathd k < \infty, \]
and for which
\begin{equation}
  \langle \varphi, \psi \rangle_{\dot{H}^1 (\mathbb{R}^2)} = \langle
  \tilde{\varphi}, \psi \rangle_{\dot{H}^1 (\mathbb{R}^2)}, \qquad \text{for
  all } \psi \in \mathcal{S} (\mathbb{R}^2) . \label{eq:sob-def}
\end{equation}
Note that the equality~{\eqref{eq:sob-def}} implies an identification of
elements whose difference is a constant. Indeed, for $C \in \mathbb{R}$,
\[ \langle \varphi, \psi \rangle_{\dot{H}^1 (\mathbb{R}^2)} = \langle
   \tilde{\varphi}, \psi \rangle_{\dot{H}^1 (\mathbb{R}^2)} = \tilde{\varphi}
   (| \mathD |^2 \psi) + C \widehat{\delta_0} (| \mathD |^2 \psi) =
   \widehat{\tilde{\varphi}} (| \cdot |^2 \hat{\psi} (\cdot)), \]
where $\mathD$ denotes the derivative operator. This means that we identify
$\varphi$ with $\tilde{\varphi} + C$ and write
\[ \| \varphi \|_{\dot{H}^1 (\mathbb{R}^2)} = \int_{\mathbb{R}^2} | k |^2  |
   \hat{\varphi} (k) |^2 \mathd k. \]
As a consequence, the tensor product $(\dot{H}^1 (\mathbb{R}^2))^{\otimes n}$,
understood as a tensor product of Hilbert spaces, can be described in the
following way: for every $\varphi \in (\dot{H}^1 (\mathbb{R}^2))^{\otimes n}$,
there exists $\tilde{\varphi} \in \mathcal{S}'_s ((\mathbb{R}^2)^n)$, the
space of symmetric tempered distributions on $(\mathbb{R}^2)^n$, such that
\[ \| \varphi \|_{(\dot{H}^1 (\mathbb{R}^2))^{\otimes n}}^2 =
   \int_{(\mathbb{R}^2)^n} \left( \prod_{i = 1}^n | k_i |^2  \right) |
   \widehat{\tilde{\varphi}} (k_{1 : n}) |^2 \mathd k_{1 : n} < \infty, \]
and for which
\[ \langle \varphi, \psi \rangle_{(\dot{H}^1 (\mathbb{R}^2))^{\otimes n}} =
   \widehat{\tilde{\varphi}} \left( \left( \prod_{i = 1}^n | \cdot_i |^2 
   \right) \hat{\psi} \right), \qquad \psi \in \mathcal{S}_s
   ((\mathbb{R}^2)^n), \]
that leaves the freedom to change $\widehat{\tilde{\varphi}}$ in $Z_n =
\bigcup_{i = 1}^n \{ k_i = 0 \} \subset (\mathbb{R}^2)^n$ and defines it
modulo a symmetric distribution $\eta$ whose Fourier transform is supported
in~$Z_n$. Therefore, we can identify $\varphi$ with $\tilde{\varphi} + \eta$,
where $\left( \prod_{i = 1}^n | k_i |^2  \right) \hat{\eta} (k_{1 : n}) = 0$.

Let us now study the operators~$\mathcal{G}_{\pm}$. With the identification
above we can define
\[ \begin{array}{lll}
     \mathcal{F} (\mathcal{G}_+^m \varphi)_n (k_{1 : n}) & = & (n - 1) \chi_m
     (k_1, k_2, k_1 + k_2) \frac{(k_1^{\perp} \cdot (k_1 + k_2)) ((k_1 + k_2)
     \cdot k_2)}{| k_1 |^2 | k_2 |^2} \hat{\varphi}_{n - 1} (k_1 + k_2, k_{3 :
     n}),\\
     \mathcal{F} (\mathcal{G}_-^m \varphi)_n (k_{1 : n}) & = & (2 \pi)^2 (n +
     1) n \int_{\mathbb{R}^2} \chi_m (p, k_1 - p, k_1) \frac{(k_1^{\perp}
     \cdot p) (k_1 \cdot q)}{| k_1 |^2} \hat{\varphi}_{n + 1} (p, k_1 - p,
     k_{2 : n}),
   \end{array} \]
with $\chi_m$ a smooth function such that $\chi_m (k_1, k_2, k_3) \approx
\mathbbm{1}_{| k_1 |, | k_2 |, | k_3 | \leqslant m}$. These formulas have to
be understood via duality
\begin{eqnarray*}
  \langle \psi, (\mathcal{G}_+^m \varphi)_n \rangle_{(\dot{H}^1)^{\otimes n}}
  & = & (n - 1) \int_{(\mathbb{R}^2)^n} \left[ \prod_{i = 1}^n | k_i |^2
  \right] \chi_m (k_1, k_2, k_1 + k_2)  \frac{(k_1^{\perp} \cdot (k_1 + k_2))
  ((k_1 + k_2) \cdot k_2)}{| k_1 |^2 | k_2 |^2} \\
  &  & \times \hat{\varphi}_{n - 1} (k_1 + k_2, k_{3 : n})  \hat{\psi}_n
  (k_{1 : n}) \mathd k_1 \cdots \mathd k_n\\
  & = & (n - 1) \int_{(\mathbb{R}^2)^n} \left[ \prod_{i = 3}^n | k_i |^2
  \right] \chi_m (k_1, k_2, k_1 + k_2)  (k_1^{\perp} \cdot (k_1 + k_2)) ((k_1
  + k_2) \cdot k_2)\\
  &  & \times \hat{\varphi}_{n - 1} (k_1 + k_2, k_{3 : n})  \hat{\psi}_n
  (k_{1 : n}) \mathd k_1 \cdots \mathd k_n .
\end{eqnarray*}
In order to check that this definition is correct, we need to make sure that
whenever $\hat{\varphi}_{n - 1}$ is supported in $Z_{n - 1}$ or when
$\hat{\psi}$ is supported in $Z_n$ the result is zero. This is obvious for
$\hat{\psi}$, so let us check it for $\hat{\varphi}_{n - 1}$. Assume
$\hat{\varphi}_{n - 1}$ is supported in $Z_{n - 1}$, then either $k_1 + k_2 =
0$ or $k_i = 0$ for some $i = 3, \ldots, n$. In the first case the result is
zero due to the multiplicative factor $(k_1^{\perp} \cdot (k_1 + k_2)) ((k_1 +
k_2) \cdot k_2)$, while in the second the result is again zero because of the
factor $\prod_{i = 3}^n | k_i |^2$. The same works for $\mathcal{G}_-$ since
it is the adjoint of~$\mathcal{G}_+$.

The expression of the norms, and therefore the results about the estimates on
the operator and so on, are exactly the same as in the periodic setting modulo
changing the sums into integrals. As already shown in Section~\ref{sec:uniq}
for the torus case, the existence and uniqueness result for the Kolmogorov
backward equation yields, via duality, uniqueness of solutions to the cylinder
martingale problem also for the case of the whole space~$\mathbb{R}^2$.

\appendix\section{Some auxiliary results}\label{s:app-A}

\begin{lemma}
  \label{lem:app-estimate}Let $C, \beta \geqslant 0$, $\alpha > (d + \beta) /
  (2 \theta)$. Then, for every $\lambda$ large,
  \[ \lambda^{- 2} \sum_{p \in \lambda^{- 1} \mathbb{Z}^d} \frac{| p
     |^{\beta}}{(| p |^{2 \theta} + | k - p |^{2 \theta} + C)^{\alpha}}
     \lesssim (| k |^{2 \theta} + C)^{(\beta + d) / (2 \theta) - \alpha},
     \qquad k \in \lambda^{- 1} \mathbb{Z}^d, \]
  uniformly in~$\lambda$.
\end{lemma}

\begin{proof}
  Since $| p |^{2 \theta} + | k - q |^{2 \theta} \gtrsim | p |^{2 \theta} + |
  k |^{2 \theta}$, we have
  \[ \begin{array}{lll}
       \lambda^{- 2} \sum_p \frac{| p |^{\beta}}{(| p |^{2 \theta} + | k - p
       |^{2 \theta} + C)^{\alpha}} & \lesssim & \int_{\mathbb{R}^d} \frac{| y
       |^{\beta}}{(| y |^{2 \theta} + | k |^{2 \theta} + C)^{\alpha}} \mathd y
     \end{array} \]
  By scaling
  \[ \int_{\mathbb{R}^d} \frac{| y |^{\beta}}{(| y |^{2 \theta} + | k |^{2
     \theta} + C)^{\alpha}} \mathd y = (| k |^{2 \theta} + C)^{(\beta + d) / 2
     \theta - \alpha} \int_{\mathbb{R}^d} \frac{| y |^{\beta}}{(| y |^{2
     \theta} + 1)^{\alpha}} \mathd y \]
  and the integral is finite if $\beta - 2 \theta \alpha < - d$.
\end{proof}

\begin{lemma}
  \label{lemma:schauder}We have, for any $T > 0$, $\gamma > 0$,
  \begin{eqnarray*}
    \sup_{0 \leqslant t \leqslant T} \| (1 +\mathcal{N})^p (1
    -\mathcal{L}_{\theta})^{1 + \gamma} \psi (t) \| & \leqslant & \| (1
    +\mathcal{N})^p (1 -\mathcal{L}_{\theta})^{1 + \gamma} \psi (0) \|\\
    &  & + \sup_{0 \leqslant t \leqslant T} \| (1 +\mathcal{N})^p (1
    -\mathcal{L}_{\theta})^{\gamma} (\partial_t - (1 -\mathcal{L}_{\theta}))
    \psi (t) \| .
  \end{eqnarray*}
\end{lemma}

\begin{proof}
  The proof is standard and proceeds by spectral calculus. Write $\Psi (t)
  \assign (\partial_t - (1 -\mathcal{L}_{\theta})) \psi (t)$
  \[ \Psi_i (s) = \mathbbm{1}_{| 1 -\mathcal{L}_{\theta} | \sim 2^i} \Psi (s),
  \]
  where $\mathbbm{1}_{| 1 -\mathcal{L}_{\theta} | \sim 2^i}$ denotes a dyadic
  partition of unity such that $\| \varphi \|^2 \approx \sum_i \|
  \mathbbm{1}_{| 1 -\mathcal{L}_{\theta} | \sim 2^i} \varphi \|^2$ for any
  $\varphi$.
  
  Let $S_t = e^{- t (1 -\mathcal{L}_{\theta})}$, so that
  \[ \psi (t) = S_t \psi (0) + \int_0^t S_{t - s} \Psi (s) \mathd s. \]
  Then, using $\| (1 -\mathcal{L}_{\theta})^{1 + \gamma} S_{t - s} \psi \|
  \lesssim ((t - s)^{- 1 - \gamma} \vee 1) \| \psi \|$ and $\| (1
  -\mathcal{L}_{\theta})^{1 + \gamma} \mathbbm{1}_{| 1 -\mathcal{L}_{\theta} |
  \sim 2^i} \| \lesssim 2^{(1 + \gamma) i}$, and letting $\delta = 2^{- i}$,
  we have
  \[ \begin{array}{lll}
       &  & \left\| (1 -\mathcal{L}_{\theta})^{1 + \gamma} \int_0^t S_{t - s}
       \Psi_i (s) \mathd s \right\|\\
       & \leqslant & \left\| (1 -\mathcal{L}_{\theta})^{1 + \gamma} \int_0^{t
       - \delta} S_{t - s} \Psi_i (s) \mathd s \right\| + \left\| (1
       -\mathcal{L}_{\theta})^{1 + \gamma} \int_{t - \delta}^t S_{t - s}
       \Psi_i (s) \mathd s \right\|\\
       & \lesssim & \int_0^{t - \delta} ((t - s)^{- 1 - \gamma} \vee 1) \|
       \Psi_i (s) \| \mathd s + 2^{(1 + \gamma) i} \int_{t - \delta}^t \| S_{t
       - s} \Psi_i (s) \| \mathd s\\
       & \lesssim & (\delta^{- \gamma} + 2^{i (1 + \gamma)} \delta) \sup_{0
       \leqslant s \leqslant T} \| \Psi_i (s) \|\\
       & \lesssim & 2^{i \gamma} \sup_{0 \leqslant s \leqslant T} \| \Psi_i
       (s) \|\\
       & \lesssim & \sup_{0 \leqslant s \leqslant T} \| (1
       -\mathcal{L}_{\theta})^{\gamma} \Psi_i (s) \|,
     \end{array} \]
  and, as a consequence,
  \[ \begin{array}{lll}
       \left\| (1 -\mathcal{L}_{\theta})^{1 + \gamma} \int_0^t S_{t - s} \Psi
       (s) \mathd s \right\|^2 & \lesssim & \sum_i \left\| (1
       -\mathcal{L}_{\theta})^{1 + \gamma} \int_0^t S_{t - s} \Psi_i (s)
       \mathd s \right\|^2\\
       & \lesssim & \sup_{0 \leqslant s \leqslant T} \sum_i \| (1
       -\mathcal{L}_{\theta})^{\gamma} \Psi_i (s) \|^2\\
       & \lesssim & \sup_{0 \leqslant s \leqslant T} \| (1
       -\mathcal{L}_{\theta})^{\gamma} \Psi (s) \|^2 .
     \end{array} \]
  Therefore, since $\mathcal{N}$ commutes with $\mathcal{L}_{\theta}$, we also
  have
  \begin{eqnarray*}
    \sup_{0 \leqslant t \leqslant T} \| (1 +\mathcal{N})^p (1
    -\mathcal{L}_{\theta})^{1 + \gamma} \psi (t) \| & \leqslant & \| (1
    +\mathcal{N})^p (1 -\mathcal{L}_{\theta})^{1 + \gamma} \psi (0) \|\\
    &  & + \sup_{0 \leqslant t \leqslant T} \| (1 +\mathcal{N})^p (1
    -\mathcal{L}_{\theta})^{\gamma} \Psi (s) \|,
  \end{eqnarray*}
  that is the claimed estimate.
\end{proof}

\section*{Acknowledgements}

It is a pleasure to contribute to the proceedings for the conference
celebrating the 60th birthday of Prof. Franco Flandoli. The first author of
this paper has learned a great deal from him, and not only about stochastic
fluid dynamics and stochastic analysis in general. The stimulating and
friendly atmosphere in Pisa was fundamental for his career. The research
exposed here is supported by DFG via CRC 1060.

\end{document}